\documentclass{amsart}
\RequirePackage{float}
\RequirePackage{graphicx}
\RequirePackage{amsmath, amsthm, amssymb, amsfonts}
\RequirePackage[colorlinks=true]{hyperref}
\hypersetup{urlcolor=blue, citecolor=blue}
\RequirePackage{subcaption}
\RequirePackage{indentfirst}

\newenvironment{Eq}{\begin{equation}\begin{aligned}}{\end{aligned}\end{equation}\ignorespacesafterend}
\newenvironment{Eq*}{\begin{equation*}\begin{aligned}}{\end{aligned}\end{equation*}\ignorespacesafterend}
\numberwithin{equation}{section}
\newtheorem{theorem}{Theorem}[section]

\newtheorem{lemma}[theorem]{Lemma}
\newtheorem{claim}[theorem]{Claim}
\theoremstyle{remark}
\newtheorem{remark}{Remark}[section] 
\theoremstyle{definition}
\newtheorem{definition}{Definition}

\DeclareMathOperator{\supp}{supp}

\DeclareMathOperator{\fd}{d}	\renewcommand{\d}{\fd}

\newcommand{\SR}{{\mathbb{R}}}
\newcommand{\SN}{{\mathbb{N}}}
\newcommand{\SZ}{{\mathbb{Z}}}

\renewcommand{\phi}{\varphi}

\newcommand{\Roma}[1]{\uppercase\expandafter{\romannumeral#1}}

\newcommand{\kl}[1]{\mathopen{}\left#1}
\newcommand{\kr}[1]{\right#1}
\renewcommand{\P}[2]{\ifthenelse{\equal{#1}{}}{\partial_{#2}}{\frac{\partial{#1}}{\partial {#2}}}}

\newcommand{\Th}[1]{Theorem \ref{#1}}
\newcommand{\Le}[1]{Lemma \ref{#1}}
\newcommand{\Se}[1]{Section \ref{#1}}
\newcommand{\Cl}[1]{Claim \ref{#1}}

\newcommand{\De}[1]{Definition \ref{#1}}
\newcommand{\Pt}[1]{Part \ref{#1}}

\newcounter{part0}[subsection]
\renewcommand{\part}[1][]{\noindent\refstepcounter{part0}{\bfseries Part \number\value{part0}:} #1.\par}
\newcounter{part1}[part0]
\newcommand{\subpart}[1][]{\noindent\stepcounter{part1}{\bfseries Part \number\value{part0}.\number\value{part1}:} #1.\par}
\newcounter{part2}[part1]

\title[Semilinear wave equations with the inverse-square potential]
{Long-time existence for semilinear wave equations with the inverse-square potential}

\author{Wei Dai}
\address{School of Mathematical Sciences\\ Zhejiang University\\ Hangzhou 310027,P.R.China}
\curraddr{Beijing International Center for Mathematical Research, Peking University, Beijing, China}
\email{daiw16@zju.edu.cn} 
\author{Daoyuan Fang}
\address{School of Mathematical Sciences\\ Zhejiang University\\ Hangzhou 310027,P.R.China}

\email{dyf@zju.edu.cn}

\author{Chengbo Wang}
\address{School of Mathematical Sciences\\ Zhejiang University\\ Hangzhou 310027,P.R.China}
\email{wangcbo@zju.edu.cn}
\urladdr{http://www.math.zju.edu.cn/wang}

\date{\today}

\begin{document}

\begin{abstract}
In this paper, we study the semilinear wave equations with the inverse-square potential. 
By transferring the original equation to a ``fractional dimensional" wave equation and analyzing the properties of its fundamental solution, we establish a long-time existence result, for sufficiently small,
spherically symmetric initial data.
Together with the previously known blow-up result, we determine the critical exponent which divides the global existence and finite time blow-up. Moreover, the sharp lower bounds of the lifespan are obtained, except for certain borderline case.
In addition, our technology allows us to handle 
an extreme case for the potential, which has hardly been discussed in literature.
\end{abstract}



\maketitle

\section{Introduction}

In this paper, 
we are interested in a kind of semilinear wave equations with the inverse-square potential
and small, spherically symmetric initial data, which has the form
\begin{Eq}\label{Eq:U_o}
\begin{cases}
\partial_t^2 U -\Delta U+Vr^{-2}U=|U|^p,  \quad r=|x|,~(t,x)\in\SR_+\times\SR^n;\\
U(0,x)=\varepsilon U_0(r),\quad U_t(0,x)=\varepsilon U_1(r);
\end{cases}
\end{Eq}
where $p>1$, 
$n\geq 2$, 
$0<\varepsilon\ll 1$ 
and $V\geq -(n-2)^2/4$ is a constant. 
We will study the long-time existence and global solvability of \eqref{Eq:U_o}. 
Specifically,  
setting $T_\varepsilon$ to be the lifespan of the solution to \eqref{Eq:U_o},
we want to know its relation with $n$, $V$, $p$ and $\varepsilon$.

When $V=0$, 
this problem reduces to the well known \emph{Strauss} conjecture, 
which has been extensively studied in a long history.
See, e.g., \cite{MR1481816}, \cite{MR1804518}, \cite{MR3247303}, \cite{MR3169791} and the references therein for more information. 
Let $p_S(n)$ be the positive root of $h_S(p;n)=0$,
where 
\begin{Eq*}
h_S(p;n):=(n-1)p^2-(n+1)p-2.
\end{Eq*}
From the early researches, 
under some natural requirements of $(U_0,U_1)$, 
it is known that
\begin{Eq*}
\kl\{\begin{aligned}
T_\varepsilon\approx &\varepsilon^{\frac{2p(p-1)}{h_S(p)}},&&\max(1,\frac{2}{n-1})<p<p_S;\\
\ln T_\varepsilon\approx& \varepsilon^{-p(p-1)},&&p=p_S;\\
T_\varepsilon=&\infty, && p>p_S.
\end{aligned}\kr.
\end{Eq*}
 Here and in what follows, 
we denote $x\lesssim y$ and $y\gtrsim x$ if $x\leq Cy$ for some $C>0$, 
independent of $\varepsilon$, 
which may change from line to line. 
We also denote $x\approx y$ if $x\lesssim y\lesssim x$.

When there exists a potential, i.e.,
 $V\neq 0$, 
the problem becomes much more complicated.
This is partly because that the inverse-square potential is in the same scaling as the wave operator,  
which means that it provides a comparable effect to the evolution of the solution.
Meanwhile, the extra singularity at the origin also needs to be taken care of.

The elliptic operator $-\Delta+V|x|^{-2}$ has been studied in several different equations related to physics and geometry, 
such as in heat equations (see, e.g., \cite{MR1760280}), in quantum mechanics (see, e.g., \cite{MR0397192}), in Schr\"odinger equations and wave equations.
Among others, the \emph{Strichartz} estimates for wave equations with the inverse square potential have been well-developed in many works.
Such result was firstly developed in \cite{MR1952384} for the wave equations with radial data.
Shortly afterwards, the radial requirement was removed by \cite{MR2003358}.
A decade later, the \emph{Strichartz} estimates with angular regularity were developed in \cite{MR3139408}.
Despite these results, we expect that these kind of estimates still have room to improve and generalize.

Turn back to the equation \eqref{Eq:U_o}. 
Note that the initial data of \eqref{Eq:U_o} are spherically symmetric,
which suggest that the solution $U$ is also spherically symmetric. 
Let $$A:=2+\sqrt{(n-2)^2+4V},~u(t,r):=r^{\frac{n-A}{2}}U(t,x).$$
A formal calculation shows that $u$ satisfies the equation
\begin{Eq}\label{Eq:u_o}
\begin{cases}
\partial_t^2 u -\Delta_Au=r^{\frac{(A-n)p+n-A}{2}}|u|^p,\quad (t,r)\in \SR_+^2,\\
u(0,x)=\varepsilon r^{\frac{n-A}{2}}U_0(r),\quad u_t(0,x)=\varepsilon r^{\frac{n-A}{2}}U_1(r),
\end{cases}
\end{Eq}
where $\Delta_A:=\partial_r^2+(A-1)r^{-1}\partial_r$.
When $A\in \SZ_+$, the operator $\Delta_A$ agrees with the $A$-dimensional Laplace operator (for radial functions),
from which we consider the parameter $A$ as the spatial ``dimension" for the equation after the transformation.

The blow-up result of \eqref{Eq:U_o} has been systematically considered in the previous paper \cite{MR4130094}  by the first author and his collaborators.
Here we define
\begin{Eq*}
h_F(p;n):=np-(n+2)
\end{Eq*}
with $p_F(n)$ be the root of $h_F(p;n)=0$, and use abbreviations
\begin{Eq*}
&p_d=p_d(A):=\frac{2}{A-1},\quad p_F=p_F((n+A-2)/2),\quad p_S=p_S(n),\\
&h_S=h_S(p;n),\quad h_F=h_F(p;(n+A-2)/2),
\end{Eq*} 
if these do not lead to ambiguity.
Then, under some requirements of initial data, there exists a constant $C=C(p;n,A)$ 
such that when $(3-A)(A+n-2)<8$, where $p_d<p_F<p_S$, we have
\begin{Eq*}
T_\varepsilon\leq 
\begin{cases}
C \varepsilon^{\frac{p-1}{h_F}},&p\leq p_d;\\
C\varepsilon^{\frac{2p(p-1)}{h_S}},&p_d<p<p_S;\\
\exp\kl(C \varepsilon^{-p(p-1)}\kr),&p=p_S.
\end{cases}
\end{Eq*}
When $(3-A)(A+n+2)=8$, where $p_d=p_F=p_S$, we have
\begin{Eq*}
T_\varepsilon\leq 
\begin{cases}
C \varepsilon^{\frac{p-1}{h_F}},&p<p_F;\\
\exp\kl(C \varepsilon^{-(p-1)}\kr),&p=p_F.
\end{cases}
\end{Eq*}
When $(3-A)(A+n+2)>8$, where $p_d>p_F>p_S$, we have
\begin{Eq*}
T_\varepsilon\leq 
\begin{cases}
C \varepsilon^{\frac{p-1}{h_F}},&p<p_F;\\
\exp\kl(C \varepsilon^{-(p-1)}\kr),&p=p_F.
\end{cases}
\end{Eq*}
This result suggests that two effects will impact the lifespan. 
For simplicity we call one \emph{Strauss} effect 
and the other \emph{Fujita} effect, since $p_S$ is the \emph{Strauss} exponent and $p_F$ is the \emph{Fujita} exponent. On the other hand, we remark that $p_F((n+A-2)/2)=p_G((n+A)/2)$, where $p_G(n)=\frac{n+1}{n-1}$ is the \emph{Glassey} exponent. 
The \emph{Glassey} exponent appears in the wave  equations with derivative nonlinearity $|\partial_t u|^p$, which suggests that there may exist some relation between the \emph{Glassey} conjecture  (see, e.g., \cite{W15}) and our problem.

For the existence part, there are also a few studies of \eqref{Eq:U_o}.
Using \emph{Strichartz} estimates, the global existence result was shown in  \cite{MR1952384, MR2003358} if
\begin{Eq*}
p\geq \frac{n+3}{n-1},\quad \frac{A-2}{2}>\frac{n-2}{2}-\frac{2}{p-1}+\max\kl\{\frac{1}{2p},\frac{1}{(n+1)(p-1)}\kr\}.
\end{Eq*}
Later, the result was further extended in \cite{MR3139408}, where  the global result in the radial case was obtained for $1+\frac{4n}{(n+1)(n-1)}<p<\frac{n+3}{n-1}$,
\begin{Eq*}
&V>\max\kl\{\frac{1}{(n-1)^2}-\frac{(n-2)^2}{4},\frac{n}{q_0}\kl(\frac{n}{q_0}-n+2\kr),\kl(\frac{n}{r_0}-n\kr)\kl(\frac{n}{r_0}-2\kr)\kr\},\\
&q_0=\frac{(p-1)(n+1)}{2},\qquad r_0=\frac{(n+1)(p-1)}{2p}.
\end{Eq*}
However, 
compared with the result of the problem without potential, in general,
it seems that the sharp result for \eqref{Eq:U_o}
could not be obtained by the \emph{Strichartz} estimates without weight.
On the other hand, there is also a gap between these results and the blow-up result we mentioned before.

Now,  
we are in a juncture to state our main results in this paper.
Firstly, we give the definition of the solution, 
and see \Se{Se:2} for further discussions. 
\begin{definition}\label{De:U_w}
We call $U$ is a weak solution of \eqref{Eq:U_o} in $[0,T]\times \SR^n$ 
if $U$ satisfies
\begin{Eq}\label{Eq:U_i}
\int_0^T\int_{\SR^n} |U|^p\Phi \d x\d t
=&\int_0^T\int_{\SR^n} U \kl(\partial_t^2-\Delta+\frac{V}{r^2}\kr)\Phi \d x\d t\\
&-\varepsilon\int_{\SR^n} (U_1\Phi(0,x)-U_0\partial_t\Phi(0,x))\d x,
\end{Eq}
for any $\Phi(t,x)\in \kl\{r^{\frac{A-n}{2}}\phi(t,x):\phi\in C_0^\infty((-\infty,T)\times \SR^n)\kr\}$. 
\end{definition}

For convenience we introduce the notations
\begin{Eq*}
p_m:=\frac{n+1}{n-1},\qquad 
p_M:=\begin{cases}\frac{n+1}{n-A}&n>A\\\infty&n\leq A\end{cases},\qquad 
p_t:=\frac{n+A}{n-1},\qquad
p_{conf}:=\frac{n+3}{n-1}.
\end{Eq*}
Then, we give the existence results for $A\in[2,3]$.
\begin{theorem}\label{Th:M_1}
Set $2\leq n$, 
$2\leq A\leq 3$ and $p_m<p<p_M$. 
Assume that the initial data satisfy
\begin{Eq}\label{Eq:ir_1}
\|r^\frac{n-A+2}{2}U_0'(r)\|_{L_r^\infty}+\|r^\frac{n-A}{2}U_0(r)\|_{L_r^\infty}+\|r^\frac{n-A+2}{2}U_1(r)\|_{L_r^\infty}<\infty,
\end{Eq} 
and supported in $[0,1)$, where $L_r^p$ stands for $L^p((0,\infty),\d r)$.
Then, there exists an $\varepsilon_0>0$ and a constant $c=c(p;n,A)$,
such that for any $0<\varepsilon<\varepsilon_0$, there is a weak solution $U$ of \eqref{Eq:U_o} in $[0,T_*)\times \SR^n$ which satisfies  
\begin{Eq*}
r^{\frac{n-A}{2}}U\in L_{loc;t,x}^\infty ([0,T_*)\times \SR^n).
\end{Eq*}
Where, when $(3-A)(A+n+2)<8$, we have $p_d<p_F<p_S$, then
\begin{Eq}\label{Eq:Main_7}
T_*=\begin{cases}
c\varepsilon^\frac{p-1}{h_F},&p<p_d;\\
c\varepsilon^\frac{p-1}{h_F} |\ln\varepsilon|^{\frac{1}{h_F}},&p=p_d;\\
c\varepsilon^\frac{2p(p-1)}{h_S},&p_d<p<p_S;\\
\exp\kl(c\varepsilon^{p(1-p)}\kr),&p=p_S;\\
\infty,&p>p_S.
\end{cases}
\end{Eq}
When $(3-A)(A+n+2)=8$, we have $p_d=p_F=p_S$, then
\begin{Eq}\label{Eq:Main_8}
T_*=\begin{cases}
c\varepsilon^\frac{p-1}{h_F},&p<p_d;\\
\exp\kl(c\varepsilon^{\frac{1-p}{2}}\kr),&p=p_d;\\
\infty,&p>p_d.
\end{cases}
\end{Eq}
When $(3-A)(A+n+2)>8$, we have $p_d>p_F>p_S$, then
\begin{Eq}\label{Eq:Main_9}
T_*=\begin{cases}
c\varepsilon^\frac{p-1}{h_F},&p<p_F;\\
\exp\kl(c\varepsilon^{1-p}\kr),&p=p_F;\\
\infty,&p>p_F.
\end{cases}
\end{Eq}
\end{theorem}

Next, we give the existence results for $A\in[3,\infty)$. 

\begin{theorem}\label{Th:M_2}
Set $2\leq n$, $A\geq 3$ and $1<p<p_{conf}$
and define $T_*$ by
\begin{Eq}\label{Eq:Main_6}
T_*=\begin{cases}
c\varepsilon^\frac{2p(p-1)}{h_S},&1<p<p_S;\\
\exp\kl(c\varepsilon^{p(1-p)}\kr),&p=p_S;\\
\infty,&p>p_S,
\end{cases}
\end{Eq}
which is the same as \eqref{Eq:Main_7} since that $p_d\leq 1$ when $A\geq3$. 

Assume that $1<p\leq p_m$ and 
the initial data $(U_0,U_1)$ satisfy
\begin{Eq}\label{Eq:ir_21}
\|r^\frac{n-1}{2}U_0(r)\|_{L_r^p}+\|r^\frac{n+1}{2}U_1(r)\|_{L_r^p}<\infty,
\end{Eq} 
and supported in $[0,1)$.
Then there exists an $\varepsilon_0>0$ and a constant $c$,
such that  for any $\varepsilon<\varepsilon_0$, 
\eqref{Eq:U_o} has a weak solution in $[0,T_*]\times \SR^n$ verifying
\begin{Eq*}
\|(1+t)^{\frac{(n-1)p-n-1}{2p}}r^{\frac{n+1}{2p}}U\|_{L_t^\infty L_r^p ([0,T_*]\times \SR_+)}<\infty,
\end{Eq*}
with $T_*$ defined in \eqref{Eq:Main_6}.

Assume that $p_m\leq p< p_S$, the initial data satisfy \eqref{Eq:ir_21} and 
\begin{Eq}\label{Eq:ir_22}
\|r^{\frac{n-1}{2}+\frac{1}{p}}U_0(r)\|_{L_r^\infty}+\|r^{\frac{n+1}{2}+\frac{1}{p}}U_1(r)\|_{L_r^\infty}<\infty,
\end{Eq}
with no compact support requirement.
Then there exists an $\varepsilon_0>0$ and a constant $c$,
such that  for any $\varepsilon<\varepsilon_0$, 
\eqref{Eq:U_o} has a weak solution in $[0,T_*]\times \SR^n$ verifying
\begin{Eq*}
\|t^{\frac{(n-1)p-n-1}{2p}}r^{\frac{n+1}{2p}}U\|_{L_t^\infty L_r^p ([0,T_*]\times \SR_+)}<\infty,
\end{Eq*}
with $T_*$ defined in \eqref{Eq:Main_6}.

Assume that $p=p_S$, 
and the initial data satisfy \eqref{Eq:ir_21} and \eqref{Eq:ir_22} for $p=p_S$ as well as some $p>p_S$.
Then there exists an $\varepsilon_0>0$ and a constant $c$,
such that  for any $\varepsilon<\varepsilon_0$, 
\eqref{Eq:U_o} has a weak solution in $[0,T_*]\times \SR^n$ verifying
\begin{Eq*}
\|r^{\frac{n+1}{2p_S}}U\|_{L_t^{p_S^2}L_r^{p_S}([0,1]\times \SR_+)}+
\|t^{\frac{1}{p_S^2}}r^{\frac{n+1}{2p_S}}U\|_{L_t^\infty L_r^{p_S} ([1,T_*]\times \SR_+)}<\infty,
\end{Eq*}
with $T_*$ defined in \eqref{Eq:Main_6}. 

Assume that $p>p_S$ and the initial data satisfy
\begin{Eq*}
\|r^\frac{n-1}{2}U_0(r)\|_{L_r^q}+\|r^\frac{n+1}{2}U_1(r)\|_{L_r^q}<\infty,\qquad q:=\frac{2(p-1)}{(n+3)-(n-1)p}.
\end{Eq*}
Then there exists an $\varepsilon_0>0$,
such that for any $\varepsilon<\varepsilon_0$, 
\eqref{Eq:U_o} has a weak solution in $\SR_+\times \SR^n$ verifying
\begin{Eq*}
\|r^{\frac{n+1}{2p}}U\|_{L_t^{pq} L_r^p(\SR_+\times \SR_+)}+\|r^{\frac{n-1}{2}}U\|_{L_t^{\infty} L_r^q(\SR_+\times \SR_+)}<\infty.
\end{Eq*}
\end{theorem}
\begin{remark}
Here we use a graph with $n\in[4,8]$ as an example
to describe the results we got. 
\begin{figure}[H]
\centering
\includegraphics[width=0.99\textwidth]{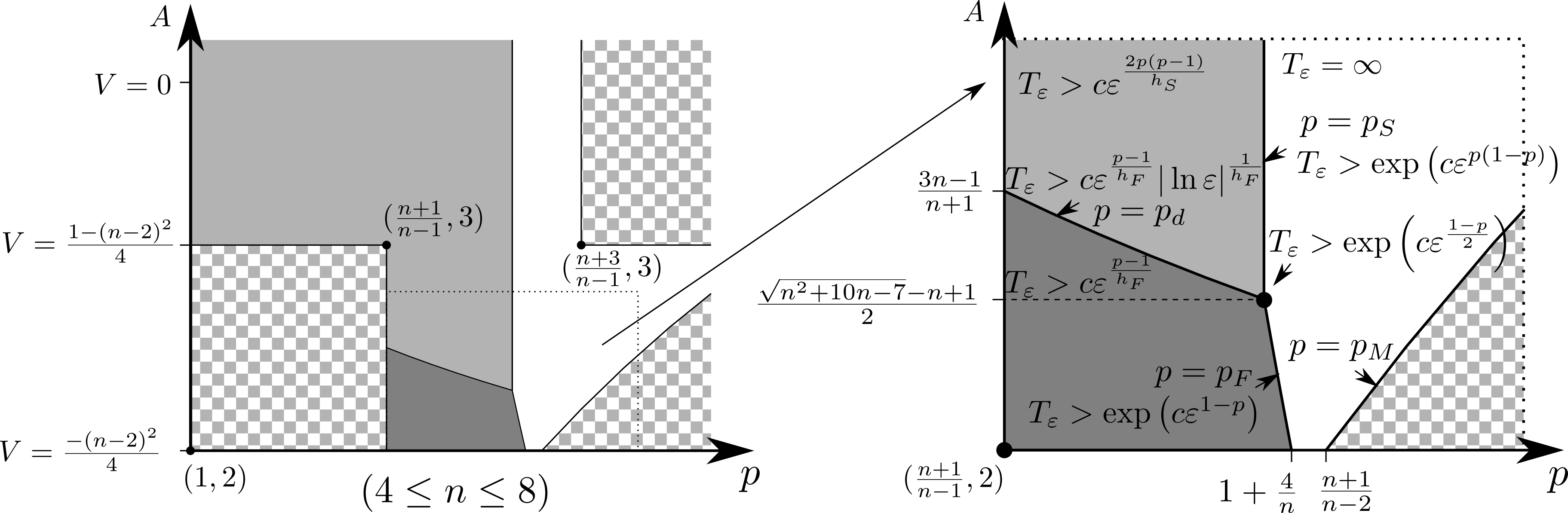}
\end{figure}
The white area stands for the region that the solution is global, 
the light gray area ($p>p_d$) stands for the region that the \emph{Strauss} effect plays role,
the dark gray area ($p<p_d$) stands for the region that \emph{Fujita} effect plays role,
and the chessboard area stands for the region that we can not deal with due to the technical difficulty.
When $n\in[2,3]$, 
we find $\frac{3n-1}{n+1}\leq 2$, 
which means that the dark gray area does not exist for $p\geq p_m=\frac{n+1}{n-1}$.
When $n=2$ we have $p_M=\infty$ for all $A\geq 2$. 
This means that the lower right chessboard area does not exist.
When $n\in[9,\infty)$, we find $1+\frac{4}{n}>\frac{n+1}{n-2}$, 
which means that the dark gray area will be slightly blocked by the lower right chessboard area.
Here we list these situations as the figures below.
\begin{figure}[H]
\centering
\includegraphics[width=0.99\textwidth]{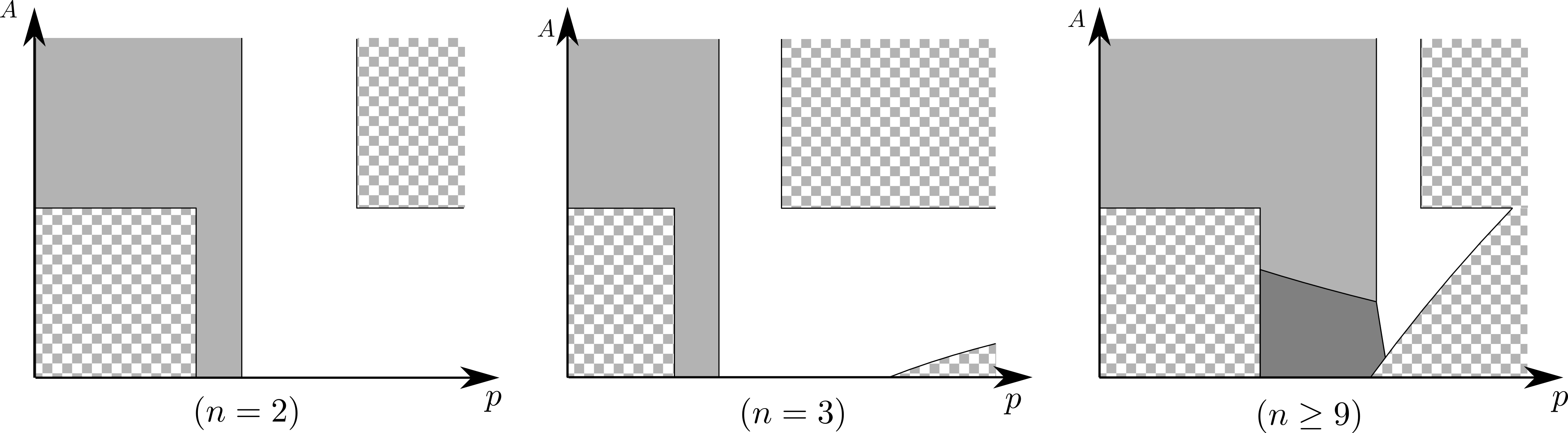}
\end{figure}
\end{remark}

\begin{remark}
The nonlinear term $|U|^p$ in \eqref{Eq:U_o} can be replaced by any $F_p(U)$ which satisfies
\begin{Eq*}
|F_p(U)|\lesssim |U|^p,\qquad |F_p(U_1)-F_p(U_2)|\lesssim |U_1-U_2|\max(|U_1|,|U_2|)^{p-1},
\end{Eq*}
and typical examples include $F_p(U)=\pm|U|^p$ and $F_p(U)=\pm|U|^{p-1}U$.
The only difference is that the constants in the result and proof need to be changed.
\end{remark}

In lower dimension, the weighted $L^\infty$ norm estimate, 
which firstly appeared in \cite{MR535704},
is very useful to prove the long-time existence result.
In \cite{MR2025737}, the authors showed long-time existence results for a two-dimensional wave system, 
where they use a trick that they take different weights in different zones.
In this paper, we further develop such method, adapted for the wave equations with potential,
and finally show the long-time existence result for $A\in[2,3]$.
On the one hand, 
our result is sharp in general,  in the sense that
our lower bound of the lifespan has the same order as the upper bound estimates as appeared in the blow-up results, except for the borderline case $p=p_d$.
On the other hand, we notice that $A=2$, which means $V=-(n-2)^2/4$, is an extreme case to the operator $-\Delta+Vr^{-2}$.
In this case, the operator is still non-negative but not positive any more, which
makes the implementation of the classical energy method more difficult.
However, our approach could handle this extreme case as well as the usual case that $V>-(n-2)^2/4$, without any additional difficulties.

In higher dimension ($n\geq 3$), 
it is well known that the weighted \emph{Strichartz} estimates 
is a helpful tool for the Strauss conjecture, particularly for the high dimensional case (see, e.g., \cite{MR1408499}, 
\cite{MR1481816} and \cite{MR1804518}).
In this paper, we adapt the approach of \cite{MR1408499} to the `fractional dimension' $A\geq 3$,
and give the long-time existence result for $A\in[3,\infty)$, which gives the sharp lower bound of the lifespan.

After comparing all the results we knew, we find that
the determination of the exact lifespan can be considered by a competition between 
the \emph{Strauss} effect and the \emph{Fujita} effect.
When $p>p_d$, the \emph{Strauss} effect is stronger, 
where the final result is only determined by the \emph{Strauss} exponent.
When $p<p_d$, the \emph{Fujita} effect is stronger, 
the final result is only determined by the \emph{Fujita} exponent.

Compared with the results for the problem without potential, 
it seems that the requirement $p>p_m$ in \Th{Th:M_1} is only a technical restriction.
Also, compared with the result of $2$-dimensional \emph{Strauss} conjecture with $p=2$ (though $2\leq p_m(2)=3$), 
we expect that the both the lower bound and the upper bound of lifespan for $p=p_d$ can be further improved.

On the other hand, it will be interesting to investigate
the problem with non-radial data, as well as the problem with more general potential functions.
It is known that when $n=3, 4$ and the potential function is of short range, the similar long-time existence results (including non-radial case) as in  
\Th{Th:M_2} are available in \cite{MW17}.
When the potential function has asymptotic behavior $Vr^{-2}$ with $V>0$, the subcritical blow-up result ($p<\max(p_S, p_F)$) was recently obtained in \cite{lai2021lifespan}.
The existence theory for the corresponding problem remains largely open.

Another interesting problem is whether or not the Cauchy problem \eqref{Eq:U_o} admits global solutions for initial data with lowest possible regularity.
For example, when $V=0$ and $2\leq n\leq 4$, such a result is available for $p\in(p_S,p_{conf}]$
with small, spherically symmetric data in the scale-invariant $(\dot H^{s_c}\times \dot H^{s_c-1})$ space ($s_c=n/2-2/(p-1)$). 
See \cite{MR2455195,MR2333654,MR2769870} for more discussion.
Here we should remark that the global result in \Th{Th:M_2} reaches the lowest regularity requirement (in the sense of scale invariance, though not in $\dot H^s$ space), but \Th{Th:M_1} still has room for improvement. 

The rest of the paper is organized as follows. In \Se{Se:2} we give a detailed discussion of \eqref{Eq:u_o} and its solution. 
In \Se{Se:3}, we restrict $A\in[2,3]$ and show the long-time existence of the solution by weighted $L^\infty$ norm estimate. 
In \Se{Se:4}, we move to situation $A\in[3,\infty)$ and establish the long-time existence result through weighted \emph{Strichartz} type estimate.

\section{Some preparations}\label{Se:2}

In this section, 
we transfer \eqref{Eq:U_o} into the equivalent equation \eqref{Eq:u_o}, 
and explain the rationality of \De{De:U_w}. 
After that, we show the formula of the solution and analyze the properties of this solution.

\subsection{The definition of weak solution}
As we said before, after introducing $u(t,r):=r^{\frac{n-A}{2}}U(t,x)$, 
a formal calculation shows that $u$ satisfied the equation \eqref{Eq:u_o}. 
We pause here and consider its linear form equation 
\begin{Eq}\label{Eq:u_l}
\begin{cases}
\partial_t^2 u -\Delta_Au=F(t,r),\quad r\in \SR_+\\
u(0,r)=f(r),\quad u_t(0,r)=g(r),
\end{cases}
\end{Eq}
with $f$, $g$, $F$ good enough. 
When $A\in \SZ_+$, 
equation \eqref{Eq:u_l} can be considered as an $A$-dimensional spherically symmetric wave equation, 
where $u(t,r)$ is a classical solution in $[0,T]$ 
if $u(t,r)$ satisfies \eqref{Eq:u_l} and $u(t,|x|)\in C^2([0,T]\times \SR^A)$. 
Thus, 
for general situation, 
we should say $u$ is a classical solution of \eqref{Eq:u_l} in $[0,T]$ 
if $u\in C^2([0,T]\times \SR_+)$, $\partial_ru(t,0)=0$ and $u$ satisfies \eqref{Eq:u_l} point wise.

Here we give a quick proof to show that such classical solution is unique. 
When $f,g,F=0$, multiplying $r^{A-1}u_t$ to both sides of \eqref{Eq:u_l} and integrating them in $\Omega:=\{(t,r):0<t<T, 0<r<R+T-t\}$, we see
\begin{Eq*}
0=\frac{1}{2}\kl.\int_0^Rr^{A-1}\kl(u_t^2+u_r^2\kr)\d r\kr|_{t=T}+\frac{1}{2\sqrt{2}}\int\limits_{t+r=T+R\atop 0<t<T}r^{A-1}(u_t-u_r)^2\d \sigma_{t,r}.
\end{Eq*}
This gives the uniqueness. 
After the discussion of the classical solution, 
we naturally say $u$ is the weak solution of \eqref{Eq:u_l} if $u$ satisfies
\begin{Eq}\label{Eq:u_i}
\int_0^T\int_0^\infty F\phi r^{A-1}\d r\d t
=&\int_0^T\int_0^\infty u \kl(\partial_t^2-\Delta_A\kr)\phi r^{A-1}\d r\d t\\
&-\int_0^\infty (g\phi(0,r)-f\partial_t\phi(0,r))r^{A-1}\d r,
\end{Eq}
for any $\phi(t,r)\in C_0^\infty\kl((-\infty,T)\times \SR\kr)$ 
with $\partial_r^{1+2k} \phi(t,0)=0$ for any $k\in \SN_0$. 
Also, set $u=r^{\frac{n-A}{2}}U$, $\phi=r^\frac{n-A}{2}\Phi$,
$f=\varepsilon r^{\frac{n-A}{2}} U_0$, $g=\varepsilon r^{\frac{n-A}{2}} U_1$ 
and $F=r^{\frac{(A-n)p+n-A}{2}}|u|^p$,
it is obvious that $u$ satisfies \eqref{Eq:u_i} is equivalent to that 
$U$ satisfies \eqref{Eq:U_i} with 
$\Phi(t,x)\in \kl\{r^{\frac{A-n}{2}}\phi(t,x):\phi\in C_0^\infty((-\infty,T)\times \SR^n)\kr\}$.
That's the reason we use \De{De:U_w} as the definition of weak solution of \eqref{Eq:U_o}.


\subsection{The formula of classical solution}
In this section we are going to give the formula of solution to \eqref{Eq:u_l}. 
We denote by $u_g$, $u_f$ and $u_F$ the solution of \eqref{Eq:u_l} with only $g\neq 0$, $f\neq 0$ and $F\neq 0$,  respectively. 
\begin{lemma}\label{Le:u_e}
Assume that $f,g, F$ are smooth enough. Then, the classical solution of \eqref{Eq:u_l} is $u=u_f+u_g+u_F$ with 
\begin{Eq*}
u_g=&r^{\frac{1-A}{2}} \int_{0}^{t+r}\rho ^{\frac{A-1}{2}}g(\rho)  I_A(\mu) \d \rho,\quad
\mu=\frac{r^2+\rho^2-t^2}{2r\rho},\\
u_f=&r^{\frac{1-A}{2}} \partial_t\int_{0}^{t+r}\rho ^{\frac{A-1}{2}}g(\rho)  I_A(\mu) \d \rho,\quad
\mu=\frac{r^2+\rho^2-t^2}{2r\rho},\\
u_F=&r^{\frac{1-A}{2}} \int_0^t\int_{0}^{t-s+r}\rho ^{\frac{A-1}{2}}F(s,\rho)  I_A(\mu) \d \rho\d s,\quad
\mu=\frac{r^2+\rho^2-(t-s)^2}{2r\rho},\\
I_A(\mu):=&\frac{2^{\frac{1-A}{2}}}{\Gamma\kl(\frac{A-1}{2}\kr)} \int_{-1}^{1} \mathcal{X}_+^{\frac{1-A}{2}}\kl(\lambda-\mu\kr)\sqrt{1-\lambda^2}^{A-3}\d \lambda.
\end{Eq*}
\end{lemma}
\begin{remark}
Here $\mathcal{X}_+^\alpha$ is a distribution, which has the expression
\begin{Eq*}
\mathcal{X}_+^\alpha(x)=\begin{cases}
0,&x<0,~\alpha>-1,\\
\frac{x^\alpha}{\Gamma(\alpha+1)},& x>0,~\alpha>-1,\\
\frac{\d~}{\d x}\mathcal{X}_+^{\alpha+1}(x), &\alpha\leq -1,\\
\end{cases}
\end{Eq*}
with $\Gamma$ the Gamma function and $\frac{\d~}{\d x}$ the weak derivative.
\end{remark}
\begin{remark}
Consider $\mu=\frac{r^2+\rho^2-t^2}{2r\rho}$. When $r> t$, we have
\begin{Eq*}
\mu|_{0<\rho<r-t}~>~\mu|_{\rho=r-t}=1~>~\mu|_{r-t<\rho<t+r}(>0)~<~\mu|_{\rho=t+r}=1~<~\mu|_{t+r<\rho},
\end{Eq*}
and when $r< t$, we have
\begin{Eq*}
\mu|_{0<\rho<t-r}~<~\mu|_{\rho=t-r}=-1~<~\mu|_{t-r<\rho<t+r}~<~\mu|_{\rho=t+r}=1~<~\mu|_{t+r<\rho}.
\end{Eq*}
By \Le{Le:I_p} below, the $\int_0^{t+r}$ in the formula of $u_g$ and $u_f$ in \Le{Le:u_e} can be replaced by $\int_{\max\{0,r-t\}}^{t+r}$, and $\int_0^{t-s+r}$ in $u_F$ can be replaced by $\int_{\max\{0,r+s-t\}}^{t-s+r}$ .
\end{remark}
To show \Le{Le:u_e}, we need to explore some properties of $I_A(\mu)$.
\begin{lemma}\label{Le:I_p}
For $A>1$ and $I_A(\mu)$ defined in \Le{Le:u_e}, we have
\begin{Eq}
I_A(1-)=\frac{1}{2},\qquad I_A(\mu)=0~for~\mu>1.\label{Eq:I_p_1}
\end{Eq}
Moreover, 
for $A\not\in 1+2\SZ_+$ with some constants $C_0$ and $C_{1}$ depending on $A$, 
we have
\begin{alignat}{2}
|\partial_{\mu}^mI_A(\mu)|\lesssim&(1-\mu)^{\frac{1-A}{2}-m}, &\qquad& \mu\leq -2,m=0,1;\label{Eq:I_p_2}\\
I_A(\mu)=&C_0\ln|1+\mu|+O(1),&\qquad& -2<\mu<1;\label{Eq:I_p_3}\\
\partial_\mu I_A(\mu)= &C_1(1+\mu)^{-1}+O\kl(|\ln|1+\mu||+1\kr),&\qquad& -2<\mu<1.\label{Eq:I_p_4}
\end{alignat}
On the other hand, for $A\in 1+2\SZ_+$, 
we have
\begin{alignat}{2}
I_A(\mu)=&0,&\qquad &\mu<-1;\label{Eq:I_p_5}\\
|\partial_\mu^{m}I_A(\mu)|\lesssim& 1,&\qquad &-1<\mu<1,~m=0,1.\label{Eq:I_p_6}
\end{alignat}
\end{lemma}

\subsection{Proof of \Le{Le:u_e}}\label{Pf_u_e}
Here we only show the proof of $u=u_g$ with $f=F=0$, the other formulas can be demonstrated by a direct calculation and \emph{Duhamel}'s principle. Without loss of generality we only deal with the case $A\not\in \SZ_+$.

\setcounter{part0}{0}
\part[Alternative expressions of $u_g$]
Before the proof, we give an alternative expressions of $u_g$ constructed in \Le{Le:u_e}. 
We first introduce a change of the variables
\begin{Eq*}
&(\rho,\lambda)=\kl(\sqrt{r^2+\tilde \rho^2-2r\tilde \rho\tilde\lambda},~\frac{r-\tilde \rho\tilde\lambda}{\sqrt{r^2+\tilde \rho^2-2r\tilde \rho\tilde\lambda}}\kr)\\
\Leftrightarrow&(\tilde\rho,\tilde\lambda)=\kl(\sqrt{r^2+ \rho^2-2r \rho\lambda},~\frac{r- \rho\lambda}{\sqrt{r^2+ \rho^2-2r \rho\lambda}}\kr).
\end{Eq*}
A direct calculation shows that the map $(\rho,\lambda)\mapsto(\tilde\rho,\tilde\lambda)$ satisfies the relation
\begin{Eq*}
\kl|\frac{\d(\rho,\lambda)}{\d(\tilde\rho,\tilde\lambda)}\kr|=\frac{\tilde\rho^2}{r^2+\tilde \rho^2-2r\tilde \rho\tilde\lambda}=\frac{\tilde\rho^2}{\rho^2},\qquad \rho^2(1-\lambda^2)=\tilde\rho^2(1-\tilde\lambda^2),
\end{Eq*}
and is a bijection from $(0,\infty)\times(-1,1)$ to itself.
For $A=1+2k+2\theta$ with $k\in \SN_0$ and $\theta\in(0,1)$, we substitute $I_A(\mu)$ into $u_g$.
Noticing $\mathcal{X}_+^{\alpha}$ is a homogeneous distribution of degree $\alpha$, we find
\begin{Eq*}
u_g=&\frac{1}{\Gamma\kl(\frac{A-1}{2}\kr)} \int_{0}^{\infty}\int_{-1}^{1} g(\rho)\mathcal{X}_+^{\frac{1-A}{2}}\kl(t^2-r^2-\rho^2+2r\rho\lambda\kr)\rho^{A-1}\sqrt{1-\lambda^2}^{A-3}\d \lambda\d\rho.\\
=&\frac{1}{\Gamma\kl(k+\theta\kr)} \int_{0}^{\infty}\int_{-1}^{1} g\kl(\sqrt{r^2+\tilde\rho^2-2r\tilde\rho\tilde\lambda}\kr)\\
&\phantom{\frac{1}{\Gamma\kl(k+\theta\kr)} \int_{0}^{\infty}\int_{-1}^{1}}
\times\mathcal{X}_+^{-k-\theta}\kl(t^2-\tilde\rho^2\kr)\tilde\rho^{A-1}\sqrt{1-\tilde\lambda^2}^{A-3}\d \tilde\lambda\d\tilde\rho\\
=&\frac{1}{\Gamma\kl(k+\theta\kr)} \kl(\frac{\partial_t}{2t}\kr)^k\int_{0}^{\infty}\int_{-1}^{1} g\kl(\sqrt{r^2+\tilde\rho^2-2r\tilde\rho\tilde\lambda}\kr)\\
&\phantom{\frac{1}{\Gamma\kl(k+\theta\kr)} \kl(\frac{\partial_t}{2t}\kr)^k\int_{0}^{\infty}\int_{-1}^{1}}
\times\mathcal{X}_+^{-\theta}\kl(t^2-\tilde\rho^2\kr)\tilde\rho^{A-1}\sqrt{1-\tilde\lambda^2}^{A-3}\d \tilde\lambda\d\tilde\rho.
\end{Eq*}
Set $\tilde\rho=t\sigma$ and $\tilde\lambda=\lambda$. Considering the definition of $\mathcal{X}_+^{-\theta}$ we finally reach
\begin{Eq}\label{Eq:u_g_c}
u_g=&\frac{1}{\Gamma\kl(k+\theta\kr)\Gamma\kl(1-\theta\kr)} \kl(\frac{\partial_t}{2t}\kr)^k\Bigg(t^{1+2k}\int_{0}^{1}\int_{-1}^{1} \frac{g\kl(\sqrt{r^2+t^2\sigma^2-2rt\sigma\lambda}\kr)}{(1-\sigma^2)^\theta}\\
&\phantom{\frac{1}{\Gamma\kl(k+\theta\kr)\Gamma\kl(1-\theta\kr)} \kl(\frac{\partial_t}{2t}\kr)^k\Bigg(t^{1+2k}\int_{0}^{1}\int_{-1}^{1}}\times\sigma^{A-1}\sqrt{1-\lambda^2}^{A-3}\d \lambda\d\sigma\Bigg).
\end{Eq}

\part[Differentiability, boundary requirement and initial requirement]
Now we begin the proof.
Firstly, using the expression we just obtained,
we can easily check that $u\in C^2(\SR_+^2)$ while $g\in C_0^\infty((0,\infty))$. 
We can also calculate that
\begin{Eq*}
\partial_ru=&\frac{1}{\Gamma\kl(k+\theta\kr)\Gamma(1-\theta)} \kl(\frac{\partial_t}{2t}\kr)^kt^{1+2k}\int_{0}^{1}\int_{-1}^{1}\frac{r-t\sigma\lambda}{\sqrt{r^2+t^2\sigma^2-2rt\sigma\lambda}} \frac{g'\kl(\sqrt{r^2+t^2\sigma^2-2rt\sigma\lambda}\kr)}{(1-\sigma^2)^\theta}\\
&\phantom{\frac{1}{\Gamma\kl(k+\theta\kr)\Gamma(1-\theta)} \kl(\frac{\partial_t}{2t}\kr)^kt^{1+2k}\int_{0}^{1}\int_{-1}^{1}}
\sigma^{A-1}\sqrt{1-\lambda^2}^{A-3}\d \lambda\d\sigma.
\end{Eq*}
Let $r=0$.
Since the integrand is an odd function of $\lambda$, 
such $u$ satisfies the boundary requirement.

To check the initial conditions we temporarily use the original expression in \Le{Le:u_e}.
Using \Le{Le:I_p} we know $I_A(\mu)=0$ when $\mu>1$, which happens when $\rho<r-t$ with $r>t$.
Then for any $r>t>0$, we have
\begin{Eq*}
u(t,r)=&r^{\frac{1-A}{2}} \int_{r-t}^{t+r}\rho^{\frac{A-1}{2}}g(\rho)  I_A\kl(\mu\kr) \d \rho\\
u_t(t,r)=&r^{\frac{1-A}{2}}\kl((t+r)^{\frac{A-1}{2}}g(t+r)+(r-t)^{\frac{A-1}{2}}g(r-t)\kr)I_A(1-)\\
&+r^{\frac{1-A}{2}} \int_{r-t}^{t+r}\rho^{\frac{A-1}{2}}g(\rho)  \frac{-t}{r\rho}I_A'\kl(\mu\kr) \d \rho.
\end{Eq*}
Let $t\rightarrow 0$. 
Using \Le{Le:I_p} again 
we find $u(0,r)=0$ and $u_t(0,r)=g(r)$. 

\part[Differential equation requirement]
Finally, 
we need to check that $u$ satisfies \eqref{Eq:u_l}. 
By a calculation trick 
\begin{Eq*}
\partial_t^2 \kl(\frac{\partial_t}{t}\kr)^{k}t^{1+2k}=\kl(\frac{\partial_t}{t}\kr)^{k+1}t^{2k+2}\partial_t,
\end{Eq*}
(see e.g. \cite[Lemma 2 in Section 2.4]{MR2597943}) we calculate that 
\begin{Eq*}
\partial_t^2 u=&\frac{2}{\Gamma\kl(k+\theta\kr)\Gamma\kl(1-\theta\kr)} \kl(\frac{\partial_t}{2t}\kr)^{k+1}\kl(t^{2k+2}w_1\kr),\\
w_1:=&\int_{0}^{1}\int_{-1}^{1}\frac{t\sigma^2-r\sigma\lambda}{\sqrt{r^2+t^2\sigma^2-2rt\sigma\lambda}} \frac{g'\kl(\sqrt{r^2+t^2\sigma^2-2rt\sigma\lambda}\kr)}{(1-\sigma^2)^\theta}\sigma^{A-1}\sqrt{1-\lambda^2}^{A-3}\d \lambda\d\sigma.
\end{Eq*}
On the other hand, a similar process as that deduced \eqref{Eq:u_g_c} also shows 
\begin{Eq*}
u=&\frac{1}{\Gamma\kl(k+\theta\kr)\Gamma\kl(2-\theta\kr)} \kl(\frac{\partial_t}{2t}\kr)^{k+1}\kl(t^{2k+2} \tilde w_2\kr),\\
\tilde w_2:=&t\int_{0}^{1}\int_{-1}^{1} \frac{g\kl(\sqrt{r^2+t^2\sigma^2-2rt\sigma\lambda}\kr)}{(1-\sigma^2)^{\theta-1}}\sigma^{A-1}\sqrt{1-\lambda^2}^{A-3}\d \lambda\d\sigma.
\end{Eq*}
Then, we see
\begin{Eq*}
\partial_r\tilde w_2=&t\int_{0}^{1}\int_{-1}^{1}\frac{r-t\sigma\lambda}{\sqrt{r^2+t^2\sigma^2-2rt\sigma\lambda}}\frac{g'\kl(\sqrt{r^2+t^2\sigma^2-2rt\sigma\lambda}\kr)}{(1-\sigma^2)^{\theta-1}}\sigma^{A-1}\sqrt{1-\lambda^2}^{A-3}\d \lambda\d\sigma\\
=&-\int_{0}^{1}\int_{-1}^{1}\frac{\partial_\lambda g\kl(\sqrt{r^2+t^2\sigma^2-2rt\sigma\lambda}\kr)}{(1-\sigma^2)^{\theta-1}}\sigma^{A-2}\sqrt{1-\lambda^2}^{A-1}\d \lambda\d\sigma\\
&-\int_{0}^{1}\int_{-1}^{1}\lambda \frac{\partial_\sigma g\kl(\sqrt{r^2+t^2\sigma^2-2rt\sigma\lambda}\kr)}{(1-\sigma^2)^{\theta-1}}\sigma^{A-1}\sqrt{1-\lambda^2}^{A-3}\d \lambda\d\sigma.
\end{Eq*}
Using integration by parts, we get
\begin{Eq*}
\partial_r\tilde w_2=&-2(1-\theta)\int_{0}^{1}\int_{-1}^{1}\lambda \frac{g\kl(\sqrt{r^2+t^2\sigma^2-2rt\sigma\lambda}\kr)}{(1-\sigma^2)^{\theta}}\sigma^{A}\sqrt{1-\lambda^2}^{A-3}\d \lambda\d\sigma.
\end{Eq*}
Thus we have
\begin{Eq*}
\partial_r u=&\frac{2}{\Gamma\kl(k+\theta\kr)\Gamma\kl(1-\theta\kr)} \kl(\frac{\partial_t}{2t}\kr)^{k+1}\kl(t^{2k+2}w_2\kr)\\
w_2:=&-\int_{0}^{1}\int_{-1}^{1}\lambda \frac{g\kl(\sqrt{r^2+t^2\sigma^2-2rt\sigma\lambda}\kr)}{(1-\sigma^2)^{\theta}}\sigma^{A}\sqrt{1-\lambda^2}^{A-3}\d \lambda\d\sigma.
\end{Eq*}
Taking the derivative again, we also have
\begin{Eq*}
\partial_r^2 u=&\frac{2}{\Gamma\kl(k+\theta\kr)\Gamma\kl(1-\theta\kr)} \kl(\frac{\partial_t}{2t}\kr)^{k+1}\kl(t^{2k+2}w_3\kr)\\
w_3:=&\int_{0}^{1}\int_{-1}^{1}\frac{t\sigma^2\lambda^2-r\sigma\lambda}{\sqrt{r^2+t^2\sigma^2-2rt\sigma\lambda}} \frac{g'\kl(\sqrt{r^2+t^2\sigma^2-2rt\sigma\lambda}\kr)}{(1-\sigma^2)^{\theta}}\sigma^{A-1}\sqrt{1-\lambda^2}^{A-3}\d \lambda\d\sigma.
\end{Eq*}
Gluing $\partial_t^2u$, $\partial_r^2 u$ and $\partial_r u$ together, we finally calculate
\begin{Eq*}
&(\partial_t^2-\partial_r^2-(A-1)r^{-1}\partial_r)u\\
=&\frac{2r^{-1}}{\Gamma\kl(k+\theta\kr)\Gamma\kl(1-\theta\kr)} \kl(\frac{\partial_t}{2t}\kr)^{k+1}\kl(t^{2k+2}\kl(rw_1-rw_3-(A-1)w_2\kr)\kr)
\end{Eq*}
where
\begin{Eq*}
rw_1-rw_3=&-\int_{0}^{1}\int_{-1}^{1}\frac{\partial_\lambda g\kl(\sqrt{r^2+t^2\sigma^2-2rt\sigma\lambda}\kr)}{(1-\sigma^2)^\theta}\sigma^{A}\sqrt{1-\lambda^2}^{A-1}\d \lambda\d\sigma\\
=&(A-1)w_2.
\end{Eq*}
This finishes the proof.

\subsection{Proof of \Le{Le:I_p}}
We begin with the second half of \eqref{Eq:I_p_1},
it is trivial since $\mathcal{X}_+^{\frac{1-A}{2}}\kl(\lambda-\mu\kr)=0$ for $\lambda\in[-1,1]$ and $\mu>1$. 
As for other results, we need to divide $A$ into two cases. 

\setcounter{part0}{0}
\part[$A$ is not odd]
We begin with the case that $A$ is not odd, i.e. $A=1+2k+2\theta$ with $k\in \SN_0$ and $0<\theta<1$. 
By definition we see that when $\mu\leq -2$ and $\lambda\in[-1,1]$, we have
\begin{Eq*}
\partial_\mu^m \mathcal{X}_+^{\frac{1-A}{2}}\kl(\lambda-\mu\kr)\approx (1-\mu)^{\frac{1-A}{2}-m},
\end{Eq*}
which gives \eqref{Eq:I_p_2}. When $-1<\mu<1$, $I_A$ has the formula
\begin{Eq*}
I_A(\mu)=&\frac{2^{-k-\theta}}{\Gamma\kl(k+\theta\kr)} \int_{-1}^{1} \partial_\lambda^k\mathcal{X}_+^{-\theta}\kl(\lambda-\mu\kr)(1-\lambda^2)^{k+\theta-1}\d \lambda\\
=&\frac{2^{-k-\theta}}{\Gamma\kl(k+\theta\kr)}\int_{-1}^{1} \mathcal{X}_+^{-\theta}\kl(\lambda-\mu\kr)(-\partial_\lambda)^k(1-\lambda^2)^{k+\theta-1}\d \lambda\\
=&\frac{2^{-k-\theta}}{\Gamma\kl(k+\theta\kr)\Gamma(1-\theta)}\int_{\mu}^{1} (\lambda-\mu)^{-\theta}\sum_{j=0}^{\lfloor k/2\rfloor }C_{j,k,\theta}\lambda^{k-2j}(1-\lambda^2)^{j+\theta-1}\d \lambda
\end{Eq*}
with some constants $C_{j,k,\theta}$. Here $\lfloor a\rfloor$ stands for the integer part of $a$. 

\subpart[$\mu$ close to $1-$]
Firstly we let $\mu$ close to $1-$. Introducing $\lambda=(1-\mu)\sigma+\mu$, we have
\begin{Eq*}
&\int_{\mu}^{1} (\lambda-\mu)^{-\theta}\lambda^{k-2j}(1-\lambda^2)^{j+\theta-1}\d \lambda\\
=&(1-\mu)^j\int_0^1 \sigma^{-\theta}(1-\sigma)^{j+\theta-1}\kl(\sigma+\mu-\mu\sigma\kr)^{k-2j}\kl(\sigma+\mu-\mu\sigma+1\kr)^{j+\theta-1}\d\sigma.
\end{Eq*} 
Let $\mu\rightarrow 1-$. Using dominated convergence theorem, we find the limit is nonzero only if $j=0$, where
\begin{Eq*}
&\lim_{\mu\rightarrow 1-}\int_0^1 \sigma^{-\theta}(1-\sigma)^{\theta-1}\kl(\sigma+\mu-\mu\sigma\kr)^{k}\kl(\sigma+\mu-\mu\sigma+1\kr)^{\theta-1}\d\sigma\\
=&2^{\theta-1}\int_0^1 \sigma^{-\theta}(1-\sigma)^{\theta-1}\d\sigma.
\end{Eq*}
Now, we calculate
\begin{Eq*}
C_{0,k,\theta}=\begin{cases}
1,&k=0;\\
2^k(k+\theta-1)(k+\theta-2)\cdots \theta,&k>0,
\end{cases}
\end{Eq*}
then
\begin{Eq*}
\lim_{\mu\rightarrow 1-}I_A(\mu)=&\frac{1}{2\Gamma\kl(\theta\kr)\Gamma(1-\theta)}\int_{0}^{1} \sigma^{-\theta}(1-\sigma)^{\theta-1}\d \sigma=\frac{1}{2}.
\end{Eq*}
This finishes the first half of \eqref{Eq:I_p_1} for non odd $A$.

For derivative, we calculate
\begin{Eq*}
&\partial_\mu\kl((1-\mu)^j\int_0^1 \sigma^{-\theta}(1-\sigma)^{j+\theta-1}\kl(\sigma+\mu-\mu\sigma\kr)^{k-2j}\kl(\sigma+\mu-\mu\sigma+1\kr)^{j+\theta-1}\d\sigma\kr)\\
=&j(1-\mu)^{j-1}\int_0^1 \sigma^{-\theta}(1-\sigma)^{j+\theta-1}\kl(\sigma+\mu-\mu\sigma\kr)^{k-2j}\kl(\sigma+\mu-\mu\sigma+1\kr)^{j+\theta-1}\d\sigma\\
&+(k-2j)(1-\mu)^j\int_0^1 \sigma^{-\theta}(1-\sigma)^{j+\theta}\kl(\sigma+\mu-\mu\sigma\kr)^{k-2j-1}\kl(\sigma+\mu-\mu\sigma+1\kr)^{j+\theta-1}\d\sigma\\
&+(j+\theta-1)\int_0^1 \sigma^{-\theta}(1-\sigma)^{j+\theta}\kl(\sigma+\mu-\mu\sigma\kr)^{k-2j}\kl(\sigma+\mu-\mu\sigma+1\kr)^{j+\theta-2}\d\sigma,
\end{Eq*}
with no singularity in all these integrals. This means $\partial_\mu I_A(\mu)=O(1)$ for $\mu$ close to $1-$, which corroborates with \eqref{Eq:I_p_4}.

\subpart[$\mu$ close to $-1+$]
Then we let $\mu$ close to $-1+$, without loss of generality we assume $\mu<-1/2$, then
\begin{Eq*}
&\int_{\mu}^{1} (\lambda-\mu)^{-\theta}\lambda^{k-2j}(1-\lambda^2)^{j+\theta-1}\d \lambda\\
=&\int_{\mu}^{0} (\lambda-\mu)^{-\theta}\lambda^{k-2j}(1-\lambda^2)^{j+\theta-1}\d \lambda+\int_{0}^{1} (\lambda-\mu)^{-\theta}\lambda^{k-2j}(1-\lambda^2)^{j+\theta-1}\d \lambda\\
=&\int_{\mu}^{0} (\lambda-\mu)^{-\theta}(1+\lambda)^{j+\theta-1}h(\lambda)\d \lambda+O(1),
\end{Eq*}
where $h(\lambda):=\lambda^{k-2j}(1-\lambda)^{j+\theta-1}$ satisfying $h(\lambda)\in C^\infty([-1,0])$. For the first integral, we split it to
\begin{Eq*}
&\int_{\mu}^{0} (\lambda-\mu)^{-\theta}(1+\lambda)^{j+\theta-1}(h(\lambda)-h(-1))\d \lambda
+h(-1)\int_{\mu}^{0} (1+\lambda)^{j-1}\d \lambda\\
&+h(-1)\kl(\int_{\mu}^{1+2\mu}+\int_{1+2\mu}^0\kr) \kl((\lambda-\mu)^{-\theta}-(1+\lambda)^{-\theta}\kr)(1+\lambda)^{j+\theta-1}\d \lambda\\
\equiv& J_1+J_2+J_3+J_4.
\end{Eq*}
Using the mean value theorem, it is easy to find that 
\begin{Eq*}
|J_1|\lesssim& \int_{\mu}^{0} (\lambda-\mu)^{-\theta}(1+\lambda)^{j+\theta}\d \lambda \lesssim 1,\\
J_2=&C_j\ln(1+\mu)+O(1),\\
|J_3|\lesssim& (1+\mu)^{j+\theta-1}\int_{\mu}^{1+2\mu} (\lambda-\mu)^{-\theta}-(1+\lambda)^{-\theta}\d \lambda\lesssim (1+\mu)^j\lesssim 1,\\
|J_4|\lesssim&(1+\mu)\int_{1+2\mu}^{0}(1+\lambda)^{j-2}\d \lambda\lesssim 1.
\end{Eq*}
Adding together, we find
\begin{Eq*}
\int_{\mu}^{1} (\lambda-\mu)^{-\theta}\lambda^{k-2j}(1-\lambda^2)^{j+\theta-1}\d \lambda
=& C_j\ln(1+\mu)+O(1),
\end{Eq*}
which gives \eqref{Eq:I_p_3} for $-1<\mu$.

As for the derivative, we introduce the change of variable $\lambda=\sigma(1+\mu)-1$, then
\begin{Eq*}
&\int_{\mu}^{1} (\lambda-\mu)^{-\theta}\lambda^{k-2j}(1-\lambda^2)^{j+\theta-1}\d \lambda\\
=&(1+\mu)^{j}\int_{1}^{(1+\mu)^{-1}} (\sigma-1)^{-\theta}\sigma^{j+\theta-1}h(\sigma(1+\mu)-1)\d \sigma\\
&+\int_{0}^{1} (\lambda-\mu)^{-\theta}\lambda^{k-2j}(1-\lambda^2)^{j+\theta-1}\d \lambda.
\end{Eq*}
Taking derivative and splitting it similarly as above, we also find
\begin{Eq*}
\partial_\mu\int_{\mu}^{1} (\lambda-\mu)^{-\theta}\lambda^{k-2j}(1-\lambda^2)^{j+\theta-1}\d \lambda=C_j'(1+\mu)^{-1}+O\kl(|\ln(1+\mu)|+1\kr), 
\end{Eq*}
which gives \eqref{Eq:I_p_4} for $-1<\mu$.

\subpart[$\mu$ close to $-1-$]
To get another part of \eqref{Eq:I_p_3}, we only need to control $I_A(\mu)-I_A(-2-\mu)$ for $-3/2< \mu<-1$. Here, for $-2\leq \mu<-1$, $I_A$ has the formula
\begin{Eq*}
I_A(\mu)=&\frac{2^{-k-\theta}}{\Gamma\kl(k+\theta\kr)\Gamma(1-\theta)}\int_{-1}^{1} (\lambda-\mu)^{-\theta}\sum_{j=0}^{\lfloor k/2 \rfloor}C_{j,k,\theta}\lambda^{k-2j}(1-\lambda^2)^{j+\theta-1}\d \lambda.
\end{Eq*}
Thus, to show \eqref{Eq:I_p_3}, we only need to estimate
\begin{Eq*}
&\int_{-1}^{1} (\lambda-\mu)^{-\theta}\lambda^{k-2j}(1-\lambda^2)^{j+\theta-1}\d \lambda-\int_{-\mu-2}^{1} (\lambda+\mu+2)^{-\theta}\lambda^{k-2j}(1-\lambda^2)^{j+\theta-1}\d \lambda\\
=&\int_{-1}^{-\mu-2} (\lambda-\mu)^{-\theta}(1+\lambda)^{j+\theta-1}h(\lambda)\d \lambda\\
&+\kl(\int_{-\mu-2}^{-2\mu-3}+\int_{-2\mu-3}^{0}\kr) \kl((\lambda-\mu)^{-\theta}-(\lambda+\mu+2)^{-\theta}\kr)(1+\lambda)^{j+\theta-1}h(\lambda)\d \lambda\\
&+\int_{0}^{1} \kl((\lambda-\mu)^{-\theta}-(\lambda+\mu+2)^{-\theta}\kr)\lambda^{k-2j}(1-\lambda^2)^{j+\theta-1}\d \lambda\\
\equiv& J_1+J_2+J_3+J_4.
\end{Eq*}
Here we have
\begin{Eq*}
|J_1|\lesssim &(-1-\mu)^{-\theta}\int_{-1}^{-\mu-2} (1+\lambda)^{j+\theta-1}|h(\lambda)|\d \lambda\lesssim 1,\\
|J_2|\lesssim &(-1-\mu)^{j+\theta-1}\int_{-\mu-2}^{-2\mu-3}(\lambda-\mu)^{-\theta}-(\lambda+\mu+2)^{-\theta}\d\lambda\lesssim(-1-\mu)^j\lesssim 1,\\
|J_3|\lesssim &(-1-\mu)\int_{-2\mu-3}^{0}(1+\lambda)^{j-2}\d\lambda\lesssim 1,\\
|J_4|\lesssim &\int_{0}^{1}(1-\lambda)^{j+\theta-1}\d\lambda\lesssim 1.\\
\end{Eq*}
In summary, we finish the proof of \eqref{Eq:I_p_3}. 

As for the derivative, we introduce the change of variable $\lambda=\mu-\sigma(1-\mu)$ for $J_2$.
A similar approach as above we find $|\partial_\mu(J_1+J_2+J_3+J_4)|\lesssim |\ln(1+\mu)|+1$.
This finishes the proof of \eqref{Eq:I_p_4}.

\part[$A$ is odd]
Next, we consider the case $A=1+2k$ with $k\in \SZ_+$. In this case we have
\begin{Eq*}
\supp \mathcal{X}_+^{\frac{1-A}{2}}\kl(x\kr)=\supp \delta^{(k-1)}\kl(x\kr)=\{0\},
\end{Eq*}
which gives \eqref{Eq:I_p_5}. On the other hand, when $-1<\mu<1$, we have
\begin{Eq*}
I_A(\mu)=&\frac{2^{-k}}{\Gamma\kl(k\kr)} \int_{-1}^{1} \partial_\lambda^{k-1}\delta\kl(\lambda-\mu\kr)(1-\lambda^2)^{k-1}\d \lambda\\
=&\frac{2^{-k}}{\Gamma\kl(k\kr)}\int_{-1}^{1} \delta\kl(\lambda-\mu\kr)(-\partial_\lambda)^{k-1}(1-\lambda^2)^{k-1}\d \lambda\\
=&\frac{2^{-k}}{\Gamma\kl(k\kr)}\sum_{j=0}^{\lfloor (k-1)/2 \rfloor}C_{j,k}'\mu^{k-1-2j}(1-\mu^2)^{j}.
\end{Eq*}
This means there is no singularity both for $I_A(\mu)$ and its derivate, which lead to \eqref{Eq:I_p_6}. 
Here we also find $C_{0,k}'=2^{k-1}(k-1)!$, 
which implies the first half of \eqref{Eq:I_p_1} for odd $A$. Now we finish the proof of \Le{Le:I_p}.

\subsection{Additional discussion of weak solutions}
In light of the fact that the framework we take is slightly different from the usual one, 
we will discuss a bit more of the weak solution. 
We will show that when 
\begin{Eq}\label{Eq:u_i_r}
r^{A-1}f\in L_{loc;r}^1,~r^{A-1}g\in L_{loc;r}^1,~ r^{A-1}F\in L_{loc;t,r}^1,~r^{A-1}u\in L_{loc;t,r}^1
\end{Eq}
with $u$ calculated by \Le{Le:u_e}, then \eqref{Eq:u_i} holds.

To show this result, 
we divide $u$ to $u_g$, $u_f$ and $u_F$.
We begin with $u_g$ part. Noticing $I_A(\mu)=0$ while $\mu>1$, in this case we have
\begin{Eq*}
&\int_0^T\int_0^\infty u\cdot \kl(\partial_t^2-\Delta_A\kr)\phi r^{A-1}\d r\d t\\
=&\int_0^T\int_0^\infty  \int_{0}^\infty r^{\frac{A-1}{2}}\rho ^{\frac{A-1}{2}}g(\rho)  I_A\kl(\frac{r^2+\rho^2-t^2}{2r\rho}\kr) \kl(\partial_t^2-\Delta_A\kr)\phi(t,r)  \d \rho\d r\d t.
\end{Eq*}
Set $t=T-s$, swap $r$ and $\rho$ then exchange the order of integration. It goes to
\begin{Eq*}
&\int_0^\infty \int_0^T \int_0^\infty r^{\frac{A-1}{2}}\rho ^{\frac{A-1}{2}}g(r)  I_A\kl(\frac{r^2+\rho^2-(T-s)^2}{2r\rho}\kr) \kl(\partial_s^2-\Delta_A\kr)\phi(T-s,\rho)  \d \rho\d s\d r.
\end{Eq*}
Here $\phi(T-s,\rho)$ has zero initial data at $s=0$ and regular enough, by expression of $u_F$ deduced in \Le{Le:u_e}, we know
\begin{Eq*}
&r^{\frac{1-A}{2}}\int_0^T \int_0^\infty \rho ^{\frac{A-1}{2}} I_A\kl(\frac{r^2+\rho^2-(T-s)^2}{2r\rho}\kr) \kl(\partial_s^2-\Delta_A\kr)\phi(T-s,\rho)  \d \rho\d s\\
=&\kl.\phi(T-s,r)\kr|_{s=T}=\phi(0,r).
\end{Eq*}
This gives \eqref{Eq:u_i} with $f=F=0$. The proof of $u_f$ and $u_F$ parts is similar, we leave them to the interested reader.

\section{Long-time existence for $A\in[2,3]$}\label{Se:3}
In this section, we will consider the case $A\in[2,3]$, and show the proof of \Th{Th:M_1}. Without loss of generality we assume $n\neq A$, otherwise $V=0$ then \eqref{Eq:U_o} reduced to the equation of \emph{Strauss} conjecture.

By the discussion in the last section, we begin to study the equation \eqref{Eq:u_o} and \eqref{Eq:u_l}.
\subsection{Estimate for homogeneous solution}
In this subsection, we will give an estimate of the homogeneous solution to \eqref{Eq:u_l}.
\begin{lemma}\label{Le:u0_e}
Let $A\in[2,3]$, $n\geq 2$ and assume $\supp(f,g)\subset [0,1)$. We have
\begin{Eq*}
|u_f+u_g|\lesssim& \kl<t+r\kr>^{\frac{1-A}{2}}\kl<t-r\kr>^{\frac{1-A}{2}}\kl(\|\rho g(\rho)\|_{L_\rho^\infty}+\|f(\rho)\|_{L_\rho^\infty}+\|\rho f'(\rho)\|_{L_\rho^\infty}\kr).
\end{Eq*}
Here and throughout the paper, $\kl<a\kr>$ stands for $\sqrt{|a|^2+4}$.
\end{lemma}

\begin{proof}[Proof of \Le{Le:u0_e}]
Here we define 
\begin{Eq*}
\Omega_0:=\{\rho:0<\rho<t-r\},\qquad \Omega_1:=\{\rho:|t-r|<\rho<\min(1,t+r)\},
\end{Eq*}
with $\Omega_0=\emptyset$ when $t<r$.

\setcounter{part0}{0}
\part[Estimate of $u_g$ with $A\in[2,3)$]
Firstly we consider $u_g$ with $A\in[2,3)$. For $u_g$, by \Le{Le:u_e} and \eqref{Eq:I_p_1} we have
\begin{Eq*}
u_g(t,r)=&r^{\frac{1-A}{2}}\kl(\int_{\Omega_0}+\int_{\Omega_1}\kr)I_A(\mu)\rho^{\frac{A-1}{2}}g(\rho)\d\rho
\equiv J_{0}+J_{1}.
\end{Eq*}
Also, for $A\in[2,3)$ we have 
\begin{Eq*}
|\ln|1+\mu||\lesssim |1+\mu|^{\frac{A-3}{2}}\lesssim |1+\mu|^{\frac{1-A}{2}},\qquad -2<\mu<1.
\end{Eq*}

\subpart[$t+r\leq 4$]
In this part, we have $\kl<t+r\kr>\approx \kl<t-r\kr>\approx 1$ and $r\lesssim 1$.
In the region of $\Omega_0$ where $\mu<-1$, 
by \eqref{Eq:I_p_2} and \eqref{Eq:I_p_3} we see
\begin{Eq*}
|I_A(\mu)|\lesssim& (-1-\mu)^{\frac{1-A}{2}}
=\kl(\frac{2r\rho}{(t+r+\rho)(t-r-\rho)}\kr)^{\frac{A-1}{2}}
\lesssim r^{\frac{A-1}{2}}(t-r-\rho)^{\frac{1-A}{2}}.
\end{Eq*}
Then we have
\begin{Eq*}
|J_{0}|\lesssim& \int_{0}^{t-r}(t-r-\rho)^{\frac{1-A}{2}} \rho^{\frac{A-3}{2}}\rho|g(\rho)|\d\rho\\
\lesssim&\kl\|\rho g(\rho)\kr\|_{L^\infty}
\lesssim\kl<t+r\kr>^{\frac{1-A}{2}}\kl<t-r\kr>^{\frac{1-A}{2}}\kl\|\rho g(\rho)\kr\|_{L^\infty}.
\end{Eq*}
In the region of $\Omega_1$ where $\mu>-1$, by \eqref{Eq:I_p_3} we see
\begin{Eq*}
|I_A(\mu)|\lesssim (1+\mu)^{\frac{A-3}{2}}
=\kl(\frac{2r\rho}{(t+r+\rho)(r+\rho-t)}\kr)^{\frac{3-A}{2}}
\lesssim \rho^{\frac{3-A}{2}}(r+\rho-t)^{\frac{A-3}{2}}.
\end{Eq*}
Thus we get
\begin{Eq*}
|J_{1}|\lesssim& r^{\frac{1-A}{2}}\int_{|t-r|}^{t+r}(r+\rho-t)^{\frac{A-3}{2}} \rho|g(\rho)|\d\rho\\
\lesssim&\kl\|\rho g(\rho)\kr\|_{L^\infty}
\lesssim\kl<t+r\kr>^{\frac{1-A}{2}}\kl<t-r\kr>^{\frac{1-A}{2}}\kl\|\rho^{\frac{A-1}{2}}g(\rho)\kr\|_{L^\infty}.
\end{Eq*}

\subpart[$t+r\geq 4$, $-1\leq t-r\leq 2$]
In this part, we have $\kl<t-r\kr>\approx 1$ and $r\approx \kl<t+r\kr>$.
In the region of $\Omega_0$ where $\mu<-1$, we have 
\begin{Eq*}
|I_A(\mu)|\lesssim& (-1-\mu)^{\frac{1-A}{2}}
=\kl(\frac{2r\rho}{(t+r+\rho)(t-r-\rho)}\kr)^{\frac{A-1}{2}}
\lesssim \rho^{\frac{A-1}{2}}(t-r-\rho)^{\frac{1-A}{2}}.
\end{Eq*}
Then we get
\begin{Eq*}
|J_{0}|\lesssim& \kl<t+r\kr>^{\frac{1-A}{2}}\int_{0}^{t-r}(t-r-\rho)^{\frac{1-A}{2}} \rho^{A-2}\rho|g(\rho)|\d\rho\\
\lesssim&\kl<t+r\kr>^{\frac{1-A}{2}}(t-r)^{\frac{A-1}{2}}\kl\|\rho g(\rho)\kr\|_{L^\infty}
\lesssim\kl<t+r\kr>^{\frac{1-A}{2}}\kl<t-r\kr>^{\frac{1-A}{2}}\kl\|\rho g(\rho)\kr\|_{L^\infty}.
\end{Eq*}
In the region of $\Omega_1$ where $\mu>-1$, we also have 
\begin{Eq*}
|I_A(\mu)|\lesssim (1+\mu)^{\frac{A-3}{2}}
=\kl(\frac{2r\rho}{(t+r+\rho)(r+\rho-t)}\kr)^{\frac{3-A}{2}}
\lesssim \rho^{\frac{3-A}{2}}(r+\rho-t)^{\frac{A-3}{2}}.
\end{Eq*}
Then we see
\begin{Eq*}
|J_{1}|\lesssim& \kl<t+r\kr>^{\frac{1-A}{2}}\int_{|t-r|}^{2}(r+\rho-t)^{\frac{A-3}{2}} \rho|g(\rho)|\d\rho\\
\lesssim&\kl<t+r\kr>^{\frac{1-A}{2}}(r+2-t)^{\frac{A-1}{2}}\kl\|\rho g(\rho)\kr\|_{L^\infty}
\lesssim\kl<t+r\kr>^{\frac{1-A}{2}}\kl<t-r\kr>^{\frac{1-A}{2}}\kl\|\rho g(\rho)\kr\|_{L^\infty}.
\end{Eq*}

\subpart[$t+r\geq 4$, $t-r\geq 2$]
In this part, we have $t+r\gtrsim\kl<t+r\kr>$ and  $t-r-1\gtrsim\kl<t-r\kr>$. Here $\Omega_1=\emptyset$ so we only need to consider $\Omega_0$ where $\mu<-1$. Again we see  
\begin{Eq*}
|I_A(\mu)|\lesssim& (-1-\mu)^{\frac{1-A}{2}}
=\kl(\frac{2r\rho}{(t+r+\rho)(t-r-\rho)}\kr)^{\frac{A-1}{2}}\\
\lesssim& r^{\frac{A-1}{2}}\rho^{\frac{A-1}{2}}\kl<t+r\kr>^{\frac{1-A}{2}}\kl<t-r\kr>^{\frac{1-A}{2}}.
\end{Eq*}
Then we have
\begin{Eq*}
|u_g|\lesssim& \kl<t+r\kr>^{\frac{1-A}{2}}\kl<t-r\kr>^{\frac{1-A}{2}}\int_{0}^{1} \rho^{A-1}|g(\rho)|\d\rho\\
\lesssim&\kl<t+r\kr>^{\frac{1-A}{2}}\kl<t-r\kr>^{\frac{1-A}{2}}\kl\|\rho g(\rho)\kr\|_{L^\infty}.
\end{Eq*}
In summary, we finish the estimate of $u_g$ when $A<3$.

\part[Estimate of $u_f$ with $A\in[2,3)$]
Next we consider $u_f$. For simplicity let $u_{g=\phi}$ stand for $u_g$ with $g=\phi$.
By \Le{Le:u_e} and the expression of $u_g$ \eqref{Eq:u_g_c}, we know
\begin{Eq*}
u_f=&\partial_t \kl(u_{g=f}\kr)\\
=&\frac{1}{\Gamma\kl(\frac{A-1}{2}\kr)\Gamma\kl(\frac{3-A}{2}\kr)} \partial_t\kl(t\int_{0}^{1}\int_{-1}^{1} \frac{f\kl(\sqrt{r^2+t^2\sigma^2-2rt\sigma\lambda}\kr)}{(1-\sigma^2)^{\frac{A-1}{2}}}\sigma^{A-1}\sqrt{1-\lambda^2}^{A-3}\d \lambda\d\sigma\kr)\\
\lesssim &t\int_{0}^{1}\int_{-1}^{1} \frac{\kl|t\sigma^2-r\sigma\lambda\kr|}{\sqrt{r^2+t^2\sigma^2-2rt\sigma\lambda}}\frac{\kl|f'\kl(\sqrt{r^2+t^2\sigma^2-2rt\sigma\lambda}\kr)\kr|}{(1-\sigma^2)^{\frac{A-1}{2}}}\sigma^{A-1}\sqrt{1-\lambda^2}^{A-3}\d \lambda\d\sigma\\
&+\int_{0}^{1}\int_{-1}^{1} \frac{\kl|f\kl(\sqrt{r^2+t^2\sigma^2-2rt\sigma\lambda}\kr)\kr|}{(1-\sigma^2)^{\frac{A-1}{2}}}\sigma^{A-1}\sqrt{1-\lambda^2}^{A-3}\d \lambda\d\sigma\\
\equiv&H_1+H_2.
\end{Eq*}
Since $A\in[2,3)$, we can easily find that $H_2\lesssim \|f\|_{L^\infty}$. Meanwhile, we find that
\begin{Eq*}
H_2\approx&t^{-1}u_{g=|f|}\lesssim t^{-1} \kl<t+r\kr>^{\frac{1-A}{2}}\kl<t-r\kr>^{\frac{1-A}{2}}\|f(\rho)\|_{L_\rho^\infty}.
\end{Eq*}
Adding together, we finish the estimate of $H_2$. Next we consider $H_1$. Noticing that when $\lambda\in[-1,1]$, we have
\begin{Eq*}
\kl(\frac{t\sigma^2-r\sigma\lambda}{\sqrt{r^2+t^2\sigma^2-2rt\sigma\lambda}}\kr)^2=&\frac{r^2\sigma^2\lambda^2+t^2\sigma^4-2rt\sigma^3\lambda}{r^2+t^2\sigma^2-2rt\sigma\lambda}\leq\sigma^2\leq 1.
\end{Eq*}
Then we calculate
\begin{Eq*}
H_1\lesssim u_{g=|f'|}\lesssim \kl<t+r\kr>^{\frac{1-A}{2}}\kl<t-r\kr>^{\frac{1-A}{2}}\kl\|\rho f'(\rho)\kr\|_{L_\rho^\infty}.
\end{Eq*}
Adding together, we finish the estimate of $u_f$ when $A<3$. 

\part[Estimate of $u_g$ and $u_f$ for $A=3$]
When $A=3$, the estimate for $u_g$ is similar, where we use \eqref{Eq:I_p_5} and \eqref{Eq:I_p_6} instead of \eqref{Eq:I_p_2} and \eqref{Eq:I_p_3}. As for $u_f$, we easily calculate 
\begin{Eq*}
u_f=\frac{(t+r)f(t+r)-(t-r)f(|t-r|)}{2r}
\end{Eq*}
with $\supp u_f\subset \{(t,r):|t-r|\leq 1\}$. When $r< t$, by mean value theorem we see 
\begin{Eq*}
|u_f|=\kl|\frac{(t+r)f(t+r)-(t-r)f(t-r)}{2r}\kr|\leq \|\partial_\rho(\rho f(\rho))\|_{L_\rho^\infty}.
\end{Eq*}
When $t<r$, we see 
\begin{Eq*}
|u_f|\leq|f(t+r)|+|f(r-t)|\lesssim \|f(\rho)\|_{L_\rho^\infty}.
\end{Eq*}
In summary, we obtain the desired estimate of $u_g$ and $u_f$ for $A=3$, and finish the proof of \Le{Le:u0_e}.
\end{proof}

\subsection{Estimate for non-homogeneous solution} 
In this subsection, we will omit the initial data and give the estimate of solution to the nonlinear equation \eqref{Eq:u_o}. For simplicity, we shift the time variable and consider the equation
\begin{Eq}\label{Eq:v_o}
\begin{cases}
\kl(\partial_t^2-\partial_r^2-(A-1)r^{-1}\partial_r\kr)v(t,r)=r^{\frac{(A-n)p+n-A}{2}}G(t,r)\\
v(4,r)=0,\qquad v_t(4,r)=0.
\end{cases}
\end{Eq}

\begin{lemma}\label{Le:v_e}
Define $\Omega:=\{(t,r)\in [4,\infty)\times \SR_+:  t>r+2\}$ and $\Lambda(t,r):=\{(s,\rho)\in\Omega:s+\rho<t+r,s-\rho<t-r\}$. We will show that, if $v$ solves the equation \eqref{Eq:v_o} with $\supp v\subset \Omega$, then for any $(t,r)\in\Omega$ and $k=1,2,3$ we have
\begin{Eq}\label{Eq:v_e}
\kl<t+r\kr>^{\frac{A-1}{2}} |v|\lesssim N_k(t-r)\|\omega_k^p G\|_{L_{s,\rho}^\infty(\Lambda)}.
\end{Eq}
Here
\begin{Eq*}
\omega_k(t,r):=&\kl<t+r\kr>^{\frac{A-1}{2}}\beta_k(t-r),\qquad k=1,2,3;\\
\beta_k(t-r):=&\begin{cases}
\kl<t-r\kr>^{\frac{A-1}{2}},&k=1,~for~ p_m<p<p_M,\\
\kl<t-r\kr>^{\frac{(n-1)p-n-1}{2}},& k=2,~for~ p_m<p<p_t,\\
\kl<t-r\kr>^{\frac{A-1}{2}}(\ln\kl<t-r\kr>)^{-1},&k=3,~for~p=p_t>p_d;
\end{cases}
\end{Eq*}
\begin{Eq*} 
N_1(t-r):=&\begin{cases}
\kl<t-r\kr>^{\frac{1-A}{2}},&p>\max(p_d,p_t)~or~p=p_d>p_t~or~ p_F<p<p_d,\\
\kl<t-r\kr>^{\frac{1-A}{2}}\ln\kl<t-r\kr>,&p=p_t>p_d~or~p=p_F<p_d,\\
\kl<t-r\kr>^\frac{(1-n)p+n+1}{2},&p_t>p>p_d,\\
\kl<t-r\kr>^\frac{(-n-A+2)p+n+3}{2},&p<\min(p_d,p_F),\\
\kl<t-r\kr>^{\frac{(-n-A+2)p+n+3}{2}}\ln\kl<t-r\kr>,& p=p_d< p_t,\\
\kl<t-r\kr>^{\frac{1-A}{2}} (\ln \kl<t-r\kr>)^2,& p=p_d= p_t;
\end{cases}\\
N_2(t-r):=&\begin{cases}
\kl<t-r\kr>^{\frac{(1-n)p+n+1}{2}}&p_S<p<p_t,\\
\kl<t-r\kr>^{\frac{(1-n)p+n+1}{2}}\ln \kl<t-r\kr>&p=p_S<p_t,\\
\kl<t-r\kr>^{\frac{(1-n)p^2+2p+n+3}{2}}&p_m<p<\min(p_S,p_t);
\end{cases}\\
N_3(t-r):=&
\kl<t-r\kr>^{\frac{1-A}{2}}\ln\kl<t-r\kr>,\qquad p=p_t>p_d.
\end{Eq*}
\end{lemma}
\begin{remark}
By the definition of $\beta_k$, we can easily find that for any $\xi, \eta>2$
with $\xi/\eta\in (1/2, 2)$, we have
\begin{Eq*}
\beta_k(\xi)\approx \beta_k(\eta).
\end{Eq*}
Also if $2<\eta_1<\eta_2$, we have
\begin{Eq*}
\beta_k(\eta_1)\lesssim\beta_k(\eta_2).
\end{Eq*}
\end{remark}

\begin{proof}[Proof of \Le{Le:v_e}]
Using \Le{Le:u_e}, we find 
\begin{Eq*}
 v=r^{\frac{1-A}{2}} \int_{\Lambda} I_A(\mu) \rho^{\frac{(A-n)p+n-1}{2}}G(s,\rho)\d\rho\d s,
\end{Eq*}
with $\mu:=\frac{r^2+\rho^2-(t-s)^2}{2 r \rho}$. To reach \eqref{Eq:v_e}, we calculate 
\begin{Eq*}
|\kl<t+r\kr>^\frac{A-1}{2} v|\leq& r^{\frac{1-A}{2}}\kl<t+r\kr>^\frac{A-1}{2} \|\omega_k^p G\|_{L_{s,\rho}^\infty (\Lambda)}\int_{\Lambda} |I_A(\mu)| \rho^{\frac{(A-n)p+n-1}{2}}\omega_k^{-p}\d\rho\d s.
\end{Eq*}
Thus, we only need to show that
\begin{Eq}\label{Eq:Jijk}
J_{ij;k}:=r^{\frac{1-A}{2}}\kl<t+r\kr>^\frac{A-1}{2} \int_{\Lambda_{ij}} |I_A(\mu)|\rho^{\frac{(A-n)p+n-1}{2}}\omega_k^{-p}\d\rho\d s\lesssim N_k(t-r)
\end{Eq}
for $i=1,2,3$, $j=1,2$ with
\begin{Eq*}
\Lambda_{11}:=&\{(s,\rho)\in\Lambda: s+\rho\in(t-r,t+r), \rho\leq s/2\};\\
\Lambda_{12}:=&\{(s,\rho)\in\Lambda: s+\rho\in(t-r,t+r), \rho\geq s/2\};\\
\Lambda_{21}:=&\{(s,\rho)\in\Lambda: s+\rho\in\kl(\frac{t-r}{2},t-r\kr), \rho\leq s/2\};\\
\Lambda_{22}:=&\{(s,\rho)\in\Lambda: s+\rho\in\kl(\frac{t-r}{2},t-r\kr), \rho\geq s/2\};\\
\Lambda_{31}:=&\{(s,\rho)\in\Lambda: s+\rho\in\kl(3,\frac{t-r}{2}\kr), \rho\leq s/2\};\\
\Lambda_{32}:=&\{(s,\rho)\in\Lambda: s+\rho\in\kl(3,\frac{t-r}{2}\kr), \rho\geq s/2\}.
\end{Eq*}

\begin{figure}[H]
\centering
\includegraphics[width=0.99\textwidth]{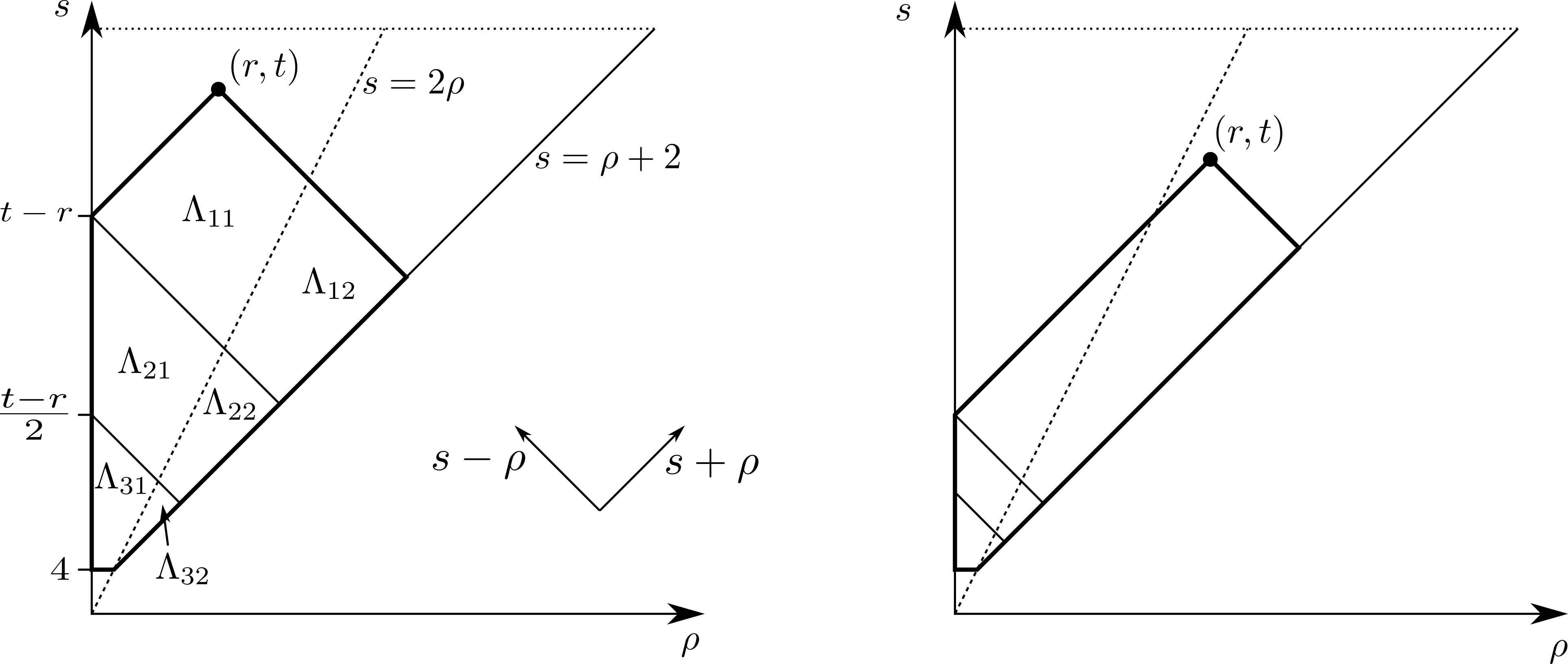}
\end{figure}
It's easy to find that
\begin{Eq}\label{Eq:srho_r}
\begin{cases}
s+\rho\leq 3(s-\rho)\leq 3(s+\rho),&(s,\rho)\in \Lambda_{11}\cup\Lambda_{21}\cup\Lambda_{31},\\
s+\rho\leq 3\rho\leq 3(s+\rho),&(s,\rho)\in \Lambda_{12}\cup\Lambda_{22}\cup\Lambda_{32}.
\end{cases}
\end{Eq}
Then, a quick calculation shows
\begin{Eq}\label{Eq:trsrho_r}
\begin{cases}
t-r\leq s+\rho \leq \min\{t+r,3(t-r)\}, \quad (t-r)/3\leq s-\rho\leq t-r, &(s,\rho)\in \Lambda_{11},\\
t-r\leq s+\rho\leq t+r,\quad 2\leq s-\rho\leq \min\{(t+r)/3,t-r\}, &(s,\rho)\in \Lambda_{12},\\
(t-r)/2\leq s+\rho \leq t-r, \quad (s+\rho)/3\leq s-\rho\leq s+\rho, &(s,\rho)\in \Lambda_{21},\\
(t-r)/2\leq s+\rho\leq t-r,\quad 2\leq s-\rho\leq (t-r)/3, &(s,\rho)\in \Lambda_{22},\\
4\leq s+\rho \leq (t-r)/2, \quad (s+\rho)/3\leq s-\rho\leq s+\rho, &(s,\rho)\in \Lambda_{31},\\
6\leq s+\rho\leq (t-r)/2,\quad 2\leq s-\rho\leq (s+\rho)/3, &(s,\rho)\in \Lambda_{32}.
\end{cases}
\end{Eq}
From now on, we introduce $\xi:=s+\rho$ and $\eta:=s-\rho$. 
We will always adopt \eqref{Eq:srho_r} and \eqref{Eq:trsrho_r} in each region.

\setcounter{part0}{0}
\part[Preparation for $A\in[2,3)$ with $r\leq t/2$]
We will firstly consider $A\in[2,3)$,
notice that $\mu=\frac{r^2+\rho^2-(t-s)^2}{2r\rho}< -1$ when $s+\rho< t-r$, and $\mu> -1$ when $s+\rho> t-r$. 
In this part we have 
\begin{Eq*}
t+r\leq 3(t-r)\leq 3(t+r).
\end{Eq*}
In the region of $\Lambda_{11}$,
using \eqref{Eq:I_p_3} we have
\begin{Eq}\label{Eq:I_e1}
|I_A(\mu)|\lesssim& (1+\mu)^{\frac{A-3}{2}}
=\kl(\frac{2r\rho}{(r+\rho+t-s)(r+\rho-t+s)}\kr)^{\frac{3-A}{2}}
\lesssim \kl(\frac{\rho}{r+\rho-t+s}\kr)^{\frac{3-A}{2}}.
\end{Eq}
Then we find
\begin{Eq*}
J_{11}\lesssim & r^{\frac{1-A}{2}} \kl<t-r\kr>^{\frac{A-1}{2}} \int_{\Lambda_{11}} \rho^{\frac{(A-n)p+n-A+2}{2}}\kl<s+\rho\kr>^{-\frac{A-1}{2}p}\beta(s+\rho)^{-p}\kl({r+\rho-t+s}\kr)^{\frac{A-3}{2}}\d\rho\d s\\
\lesssim &r^{\frac{1-A}{2}}\kl<t-r\kr>^{\frac{(1-A)p+A-1}{2}}\beta(t-r)^{-p}\int_{t-r}^{t+r}\int_{(t-r)/3}^{t-r} (\xi-\eta)^{\frac{(A-n)p+n-A+2}{2}}\kl({\xi-(t-r)}\kr)^{\frac{A-3}{2}}\d \eta\d\xi\\
\lesssim& r^{\frac{1-A}{2}}\kl<t-r\kr>^{\frac{(1-n)p+n+3}{2}}\beta(t-r)^{-p}\int_{t-r}^{t+r}\kl({\xi-(t-r)}\kr)^{\frac{A-3}{2}}\d\xi\\
\lesssim&\kl<t-r\kr>^{\frac{(1-n)p+n+3}{2}}\beta(t-r)^{-p},
\end{Eq*}
where we noticed $\frac{(A-n)p+n-A+2}{2}> -1$ while $p<p_M$.
In the region of $\Lambda_{12}$, we also have \eqref{Eq:I_e1}. Then we find
\begin{Eq*}
J_{12}\lesssim & r^{\frac{1-A}{2}} \kl<t-r\kr>^{\frac{A-1}{2}}\int_{\Lambda_{12}}\kl<s+\rho\kr>^{\frac{(1-n)p+n-A+2}{2}}\beta(s-\rho)^{-p}\kl({r+\rho-t+s}\kr)^{\frac{A-3}{2}}\d\rho\d s\\
\lesssim & r^{\frac{1-A}{2}}\kl<t-r\kr>^{\frac{(1-n)p+n+1}{2}}\int_{t-r}^{t+r}\kl({\xi-(t-r)}\kr)^{\frac{A-3}{2}}\d\xi \int_2^{(t+r)/3} \beta(\eta)^{-p}\d\eta\\
\lesssim &\kl<t-r\kr>^{\frac{(1-n)p+n+1}{2}} \int_2^{t-r} \beta(\eta)^{-p}\d\eta.
\end{Eq*}
In the region of $\Lambda_{21}$,
using \eqref{Eq:I_p_2} and \eqref{Eq:I_p_3}
we have
\begin{Eq}\label{Eq:I_e2}
|I_A(\mu)|\lesssim& (-1-\mu)^{\frac{1-A}{2}}
=\kl(\frac{2r\rho}{(t+r-s+\rho)(t-r-s-\rho)}\kr)^{\frac{A-1}{2}}
\lesssim \kl(\frac{r}{t-r-s-\rho}\kr)^{\frac{A-1}{2}}.
\end{Eq}
Then we find
\begin{Eq*}
J_{21}\lesssim & \kl<t-r\kr>^{\frac{A-1}{2}} \int_{\Lambda_{21}} \rho^{\frac{(A-n)p+n-1}{2}}\kl<s+\rho\kr>^{\frac{(1-A)p}{2}}\beta(s+\rho)^{-p}\kl({t-r-s-\rho}\kr)^{\frac{1-A}{2}}\d\rho\d s\\
\lesssim &\kl<t-r\kr>^{\frac{(1-A)p+A-1}{2}}\beta(t-r)^{-p} \int_{(t-r)/2}^{t-r}\int_{\xi/3}^{\xi} (\xi-\eta)^{\frac{(A-n)p+n-1}{2}}\kl({t-r-\xi}\kr)^{\frac{1-A}{2}}\d \eta\d\xi\\
\lesssim &\kl<t-r\kr>^{\frac{(1-n)p+n+A}{2}}\beta(t-r)^{-p} \int_{(t-r)/2}^{t-r} \kl(t-r-\xi\kr)^{\frac{1-A}{2}}\d\xi\\
\lesssim&\kl<t-r\kr>^{\frac{(1-n)p+n+3}{2}}\beta(t-r)^{-p},
\end{Eq*}
where we require $p<p_M$ to ensure that $\frac{(A-n)+n-1}{2}> -1$.
In the region of $\Lambda_{22}$, we still have 
\eqref{Eq:I_e2}. Then we find
\begin{Eq*}
J_{22}\lesssim & \kl<t-r\kr>^{\frac{A-1}{2}}  \int_{\Lambda_{22}}\kl<s+\rho\kr>^{\frac{(1-n)p+n-1}{2}}\beta(s-\rho)^{-p}\kl({t-r-s-\rho}\kr)^{\frac{1-A}{2}}\d\rho\d s\\
\lesssim & \kl<t-r\kr>^{\frac{(1-n)p+n+A-2}{2}}\int_{(t-r)/2}^{t-r}\kl(t-r-\xi\kr)^{\frac{1-A}{2}}\d\xi \int_2^{(t-r)/3}\beta(\eta)^{-p}\d\eta\\
\lesssim &\kl<t-r\kr>^{\frac{(1-n)p+n+1}{2}} \int_2^{(t-r)/3} \beta(\eta)^{-p}\d\eta.
\end{Eq*}
In the region of $\Lambda_{31}$
we have 
\begin{Eq}\label{Eq:I_e3}
|I_A(\mu)|\lesssim& (-1-\mu)^{\frac{1-A}{2}}
=\kl(\frac{2r\rho}{(t+r-s+\rho)(t-r-s-\rho)}\kr)^{\frac{A-1}{2}}
\lesssim\kl(\frac{r\rho}{(t-r)^2}\kr)^{\frac{A-1}{2}}.
\end{Eq}
Then we find
\begin{Eq*}
J_{31}\lesssim &\kl<t-r\kr>^{\frac{A-1}{2}} \int_{\Lambda_{31}} \rho^{\frac{(A-n)p+n+A-2}{2}}\kl<s+\rho\kr>^{\frac{(1-A)p}{2}}\beta(s+\rho)^{-p}\kl({t-r}\kr)^{-A+1}\d\rho\d s\\
\lesssim &\kl<t-r\kr>^{\frac{1-A}{2}}\int_4^{(t-r)/2}\int_{\xi/3}^{\xi} (\xi-\eta)^{\frac{(A-n)p+n+A-2}{2}}\kl<\xi\kr>^{\frac{(1-A)p}{2}}\beta(\xi)^{-p}\d \eta\d\xi\\
\lesssim &\kl<t-r\kr>^{\frac{1-A}{2}}\int_4^{(t-r)/2}\kl<\xi\kr>^{\frac{(1-n)p+n+A}{2}}\beta(\xi)^{-p}  \d\xi,
\end{Eq*} 
where we noticed $\frac{(A-n)p+n+A-2}{2}>-1$ while $p<p_M$.
In the region of $\Lambda_{32}$ we still have \eqref{Eq:I_e3}.
Then we find
\begin{Eq*}
J_{32}\lesssim &\kl<t-r\kr>^{\frac{A-1}{2}} \int_{\Lambda_{32}}\kl<s+\rho\kr>^{\frac{(1-n)p+n+A-2}{2}}\beta(s-\rho)^{-p}\kl({t-r}\kr)^{-A+1}\d\rho\d s\\
\lesssim &\kl<t-r\kr>^{\frac{1-A}{2}}\int_6^{(t-r)/2}\int_2^{\xi/3}\kl<\xi\kr>^{\frac{(1-n)p+n+A-2}{2}}\beta(\eta)^{-p}\d \eta\d\xi.
\end{Eq*}

\part[Preparation for $A\in[2,3)$ with $r\geq t/2$]
In this part we have
\begin{Eq*}
t+r\leq 3r\leq 3(t+r).
\end{Eq*}
In the region of $\Lambda_{11}$, we have \eqref{Eq:I_e1}.
Then we find
\begin{Eq*}
J_{11}'\lesssim & \int_{\Lambda_{11}} \rho^{\frac{(A-n)p+n-A+2}{2}}\kl<s+\rho\kr>^{\frac{(1-A)p}{2}}\beta(s+\rho)^{-p}\kl({r+\rho-t+s}\kr)^{\frac{A-3}{2}}\d\rho\d s\\
\lesssim &\kl<t-r\kr>^{\frac{(1-A)p}{2}}\beta(t-r)^{-p} \int_{t-r}^{3(t-r)}\int_{(t-r)/3}^{t-r} (\xi-\eta)^{\frac{(A-n)p+n-A+2}{2}}\kl({\xi-(t-r)}\kr)^{\frac{A-3}{2}}\d \eta\d\xi\\
\lesssim &\kl<t-r\kr>^{\frac{(1-n)p+n-A+4}{2}}\beta(t-r)^{-p} \int_{t-r}^{3(t-r)}\kl({\xi-(t-r)}\kr)^{\frac{A-3}{2}}\d\xi\\
\lesssim&\kl<t-r\kr>^{\frac{(1-n)p+n+3}{2}}\beta(t-r)^{-p},
\end{Eq*}
again $\frac{(A-n)p+n-A+2}{2}> -1$ since $p<p_M$.
In the region of $\Lambda_{12}$, we have \eqref{Eq:I_e1}.
Then we find
\begin{Eq*}
J_{12}'\lesssim &   \int_{\Lambda_{12}}\kl<s+\rho\kr>^{\frac{(1-n)p+n-A+2}{2}}\beta(s-\rho)^{-p} \kl({r+\rho-t+s}\kr)^{\frac{A-3}{2}}\d\rho\d s\\
\leq &   \int_{t-r}^{t+r}\kl<\xi\kr>^{\frac{(1-n)p+n-A+2}{2}}\kl({\xi-(t-r)}\kr)^{\frac{A-3}{2}}\d\xi \int_2^{t-r} \beta(\eta)^{-p}\d\eta\\
\lesssim &  
 \kl(\kl<t-r\kr>^{\frac{(1-n)p+n+1}{2}} +\int_{3(t-r)}^{t+r}\kl<\xi\kr>^{\frac{(1-n)p+n-1}{2}}\d\xi\kr)\int_2^{t-r} \beta(\eta)^{-p}\d\eta\\
 \lesssim & \kl<t-r\kr>^{\frac{(1-n)p+n+1}{2}} \int_2^{t-r} \beta(\eta)^{-p}\d\eta,
\end{Eq*}
where we require $p>p_m$ so that $\frac{(1-n)p+n-1}{2}<-1$.
In the region of $\Lambda_{21}$ we have
\begin{Eq}\label{Eq:I_e4}
|I_A(\mu)|\lesssim& (-1-\mu)^{\frac{1-A}{2}}
=\kl(\frac{2r\rho}{(t+r-s+\rho)(t-r-s-\rho)}\kr)^{\frac{A-1}{2}}
\lesssim \kl(\frac{\rho}{t-r-s-\rho}\kr)^{\frac{A-1}{2}}.
\end{Eq}
Then we find
\begin{Eq*}
J_{21}'\lesssim & \int_{\Lambda_{21}} \rho^{\frac{(A-n)p+n+A-2}{2}}\kl<s+\rho\kr>^{\frac{(1-A)p}{2}}\beta(s+\rho)^{-p}\kl({t-r-s-\rho}\kr)^{\frac{1-A}{2}}\d\rho\d s\\
\leq &\kl<t-r\kr>^{\frac{(1-A)p}{2}}\beta(t-r)^{-p}\int_{(t-r)/2}^{t-r}\int_{\xi/3}^{\xi} (\xi-\eta)^{\frac{(A-n)p+n+A-2}{2}}\kl({t-r-\xi}\kr)^{\frac{1-A}{2}}\d \eta\d\xi\\
\lesssim& \kl<t-r\kr>^{\frac{(1-n)p+n+A}{2}}\beta(t-r)^{-p}\int_{(t-r)/2}^{t-r} \kl(t-r-\xi\kr)^{\frac{1-A}{2}}\d\xi\\
\lesssim&  \kl<t-r\kr>^{\frac{(1-n)p+n+3}{2}}\beta(t-r)^{-p},
\end{Eq*}
where $\frac{(A-n)p+n+A-2}{2}>-1$ since $p<p_M$.
In the region of $\Lambda_{22}$, we have \eqref{Eq:I_e4}. 
Then we find
\begin{Eq*}
J_{22}'\lesssim &  \int_{\Lambda_{22}}\kl<s+\rho\kr>^{\frac{(1-n)p+n+A-2}{2}}\beta(s-\rho)^{-p}\kl({t-r-s-\rho}\kr)^{\frac{1-A}{2}}\d\rho\d s\\
\leq &\kl<t-r\kr>^{\frac{(1-n)p+n+A-2}{2}}\int_{(t-r)/2}^{t-r}\kl(t-r-\xi\kr)^{\frac{1-A}{2}}\d\xi \int_2^{(t-r)/3} \beta(\eta)^{-p}\d\eta\\
\lesssim &\kl<t-r\kr>^{\frac{(1-n)p+n+1}{2}} \int_2^{(t-r)/3} \beta(\eta)^{-p}\d\eta.
\end{Eq*}
In the region of $\Lambda_{31}$, we have
\begin{Eq}\label{Eq:I_e5}
|I_A(\mu)|\lesssim& (-1-\mu)^{\frac{1-A}{2}}
=\kl(\frac{2r\rho}{(t+r-s+\rho)(t-r-s-\rho)}\kr)^{\frac{A-1}{2}}
\lesssim \kl(\frac{\rho}{t-r}\kr)^{\frac{A-1}{2}}.
\end{Eq}
Then we find
\begin{Eq*}
J_{31}'\lesssim & \int_{\Lambda_{31}} \rho^{\frac{(A-n)p+n+A-2}{2}}\kl<s+\rho\kr>^{\frac{(1-A)p}{2}}\beta(s+\rho)^{-p}\kl({t-r}\kr)^{\frac{1-A}{2}}\d\rho\d s\\
\lesssim &\kl<t-r\kr>^{\frac{1-A}{2}} \int_4^{(t-r)/2}\int_{\xi/3}^{\xi} (\xi-\eta)^{\frac{(A-n)p+n+A-2}{2}}\kl<\xi\kr>^{\frac{(1-A)p}{2}}\beta(\xi)^{-p}\d \eta\d\xi\\
\lesssim &\kl<t-r\kr>^{\frac{1-A}{2}} \int_4^{(t-r)/2}\kl<\xi\kr>^{\frac{(1-n)p+n+A}{2}}\beta(\xi)^{-p}\d\xi,
\end{Eq*} 
where $\frac{(A-n)p+n+A-2}{2}>-1$ since $p<p_M$.
In the region of $\Lambda_{32}$, we have \eqref{Eq:I_e5}. 
Then we find
\begin{Eq*}
J_{32}'\lesssim&  \int_{\Lambda_{32}}\kl<s+\rho\kr>^{\frac{(1-n)p+n+A-2}{2}}\beta(s-\rho)^{-p}\kl({t-r}\kr)^{\frac{1-A}{2}}\d\rho\d s\\
\lesssim &\kl<t-r\kr>^{\frac{1-A}{2}}\int_6^{(t-r)/2}\int_2^{\xi/3}\kl<\xi\kr>^{\frac{(1-n)p+n+A-2}{2}}\beta(\eta)^{-p} \d \eta\d\xi.
\end{Eq*} 

\part[Estimate for $A\in[2,3)$]
Turning to the proof of \eqref{Eq:Jijk},
we will only present the estimate of $J_{32}$,
and the other terms can be estimated in a similar manner. 
At first,
for $J_{32;1}$ with $p_m<p<p_M$, 
we have
\begin{Eq*}
J_{32;1}\lesssim \kl<t-r\kr>^{\frac{1-A}{2}}\int_6^{(t-r)/2}\int_2^{\xi/3}\kl<\xi\kr>^{\frac{(1-n)p+n+A-2}{2}}\kl<\eta\kr>^{-\frac{A-1}{2}p}\d\eta\d\xi.
\end{Eq*}
When $p>p_d=2/(A-1)$, 
it is easy to see that
\begin{Eq*}
J_{32;1}\lesssim \kl<t-r\kr>^{\frac{1-A}{2}}\int_6^{(t-r)/2}\kl<\xi\kr>^{\frac{(1-n)p+n+A-2}{2}}\d\xi \lesssim N_1(t-r).
\end{Eq*}
Similarly, when $p\le p_d$, 
we have
\begin{Eq*}
J_{32;1}\lesssim
\begin{cases}
\kl<t-r\kr>^{\frac{1-A}{2}}\int_6^{(t-r)/2}\kl<\xi\kr>^{\frac{(1-n)p+n+A-2}{2}}\ln\xi\d\xi, & p=p_d,\\
\kl<t-r\kr>^{\frac{1-A}{2}}\int_6^{(t-r)/2}\kl<\xi\kr>^{\frac{(-n-A+2)p+n+A}{2}}\d\xi,& p<p_d,
\end{cases}
\end{Eq*} which are controlled by $N_1(t-r)$ and
this finishes the proof of \eqref{Eq:v_e} with $k=1$. 
For $J_{32;2}$ with $p_m<p<p_t$, we have
\begin{Eq*}
J_{32;2}\lesssim& \kl<t-r\kr>^{\frac{1-A}{2}}\int_6^{(t-r)/2}\kl<\xi\kr>^{\frac{(1-n)p+n+A-2}{2}}\d\xi\int_2^{(t-r)/6}\kl<\eta\kr>^{\frac{(1-n)p^2+(n+1)p}{2}} \d \eta\\
\lesssim&\kl<t-r\kr>^{\frac{(1-n)p+n+1}{2}}\int_2^{(t-r)/6}\kl<\eta\kr>^{\frac{(1-n)p^2+(n+1)p}{2}} \d \eta\lesssim N_2(t-r).
\end{Eq*}
Finally, for $J_{32;3}$ with $p=p_t>p_d$, we have
\begin{Eq*}
J_{32;3}\lesssim& \kl<t-r\kr>^{\frac{1-A}{2}}\int_6^{(t-r)/2}\int_2^{\xi/3}\kl<\xi\kr>^{-1}\kl<\eta\kr>^{\frac{(1-A)p}{2}} (\ln\eta)^{p}\d \eta\d\xi\\
\lesssim& \kl<t-r\kr>^{\frac{1-A}{2}}\int_6^{(t-r)/2}\kl<\xi\kr>^{-1} \d\xi\lesssim \kl<t-r\kr>^{\frac{1-A}{2}}\ln\kl<t-r\kr>= N_3(t-r).
\end{Eq*}
In conclusion, this completes the proof for $A\in[2,3)$. 

\part[Estimate for $A=3$]
The case $A=3$ is much simpler, thanks to \eqref{Eq:I_p_5},
we only need to consider $\Lambda_{11}$ and $\Lambda_{12}$. 
By \eqref{Eq:I_p_6} and a similar approach as above we get the desired estimate.
\end{proof}

\subsection{Long-time existence}
In this subsection, 
we will construct a Cauchy sequence to approximate the desired solution. 
We set $u_{-1}=0$ and let $u_{j+1}$ be the solution of the equation
\begin{Eq}\label{Eq:uj_o}
\begin{cases}
\partial_t^2 u_{j+1} -\Delta_Au_{j+1}=r^{\frac{(A-n)p+n-A}{2}}|u_{j}|^p,\quad r\in \SR_+,\\
u_{j+1}(0,x)=\varepsilon r^{\frac{n-A}{2}}U_0(r),\quad \partial_t u_{j+1}(0,x)=\varepsilon r^{\frac{n-A}{2}}U_1(r).
\end{cases}
\end{Eq}
By \Le{Le:u0_e} and \Le{Le:v_e}, 
noticing that for any $p>1$, we have
\begin{Eq*}
\kl||a|^p-|b|^p\kr|\lesssim |a-b|\max(|a|,|b|)^{p-1},
\end{Eq*}
then we see
\begin{Eq*}
\kl<t+r\kr>^\frac{A-1}{2}|u_{j+1}|\leq& \varepsilon C_0\kl<t-r\kr>^\frac{1-A}{2}\Psi+C_0N_k(t-r)\|\omega_k u_j\|_{L_{s,\rho}^\infty(\Lambda)}^p,\\
\kl<t+r\kr>^\frac{A-1}{2}|u_{j+1}-u_{j}|\leq& C_0N_k(t-r)\|\omega_k (u_j-u_{j-1})\|_{L_{s,\rho}^\infty(\Lambda)}\max_{l\in\{j,j-1\}}\|\omega_k u_l\|_{L_{s,\rho}^\infty(\Lambda)}^{p-1},
\end{Eq*}
with $k=1,2,3$ and $C_0$ large enough.
Here 
\begin{Eq*}
\Psi:=&\|r^\frac{n-A+2}{2}U_0'(r)\|_{L_r^\infty}+\|r^\frac{n-A}{2}U_0(r)\|_{L_r^\infty}+\|r^\frac{n-A+2}{2}U_1(r)\|_{L_r^\infty},\\
\Lambda(t,r):=&\kl\{(s,\rho)\in\Omega: s+\rho<t+r,s-\rho<t-r\kr\},\\
\Omega:=&\kl\{(t,r)\in \SR_+^2:t>r-1\kr\}.
\end{Eq*}

To prove \Th{Th:M_1}, we need to separate $p\in(p_m,p_M)$ into more parts rather than that in \eqref{Eq:Main_7}, \eqref{Eq:Main_8} or \eqref{Eq:Main_9}. For the reader's convenience, we list them as below. When $(3-A)(A+n+2)<8$, we have $p_d<p_F<p_S<p_t$. Then the proof for $p<p_d$ will be found in \Pt{Pt:4}, $p=p_d$ in \Pt{Pt:6}, $p_d<p<p_S$  in \Pt{Pt:8}, $p=p_S$ in \Pt{Pt:9}, $p_S<p<p_t$ in \Pt{Pt:3}, $p=p_t$ in \Pt{Pt:2} and $p>p_t$ in \Pt{Pt:1}. When $(3-A)(A+n+2)=8$, we have $p_d=p_F=p_S=p_t$. The proof for $p<p_d$ will be found in \Pt{Pt:4}, $p=p_d$ in \Pt{Pt:7} and $p>p_d$ in \Pt{Pt:1}. Finally when $(3-A)(A+n+2)>8$, we have $p_d>p_F>p_S>p_t$. The proof for $p<p_F$ will be found in \Pt{Pt:4}, $p=p_F$ in \Pt{Pt:5} and $p>p_F$ in \Pt{Pt:1}.

Now, we are prepared to give the proofs for each part.

\setcounter{part0}{0}
\part[$\max(p_t,p_F)<p$]\label{Pt:1}
In this part, we choose $k=1$. For $(t,r)\in\Omega$ we find
\begin{Eq*}
\omega_1|u_{j+1}|\leq& \varepsilon C_0\Psi+C_0\|\omega_1 u_j\|_{L_{s,\rho}^\infty(\Lambda)}^p,\\
\omega_1|u_{j+1}-u_j|\leq& C_0\|\omega_1 \kl(u_j-u_{j-1}\kr)\|_{L_{s,\rho}^\infty(\Lambda)}\max_{l\in\{j,j-1\}}\|\omega_1 u_l\|_{L_{s,\rho}^\infty(\Lambda)}^{p-1}.
\end{Eq*}
Taking the $L_{t,r}^\infty(\Omega)$ norm on both sides and we get
\begin{Eq*}
\kl\|\omega_1u_{j+1}\kr\|_{L_{t,r}^\infty(\Omega)}\leq& \varepsilon C_0\Psi+C_0\|\omega_1 u_j\|_{L_{t,r}^\infty(\Omega)}^p,\\
\kl\|\omega_1(u_{j+1}-u_j)\kr\|_{L_{t,r}^\infty(\Omega)}\leq& C_0\|\omega_1 \kl(u_j-u_{j-1}\kr)\|_{L_{t,r}^\infty(\Omega)}\max_{l\in\{j,j-1\}}\|\omega_1 u_l\|_{L_{t,r}^\infty(\Omega)}^{p-1}.
\end{Eq*}
For any $\varepsilon>0$ satisfying $(2\varepsilon C_0\Psi)^p<\varepsilon \Psi$, we find
\begin{Eq*}
\kl\|\omega_1u_{j}\kr\|_{L_{t,r}^\infty(\Omega)}\leq 2\varepsilon C_0\Psi
\end{Eq*}
holds for any $j$ since $u_{-1}=0$. Meanwhile, it also gives us
\begin{Eq*}
\kl\|\omega_1(u_{j+1}-u_j)\kr\|_{L_{t,r}^\infty(\Omega)}\leq& C_0\kl(2\varepsilon C_0\Psi\kr)^{p-1}\|\omega_1 \kl(u_j-u_{j-1}\kr)\|_{L_{t,r}^\infty(\Omega)}\\
\leq &\frac{1}{2}\|\omega_1 \kl(u_j-u_{j-1}\kr)\|_{L_{t,r}^\infty(\Omega)}.
\end{Eq*}
This means $\{u_j\}$ is a Cauchy sequence in weighted $L^\infty$ norm. 
Set the limit as $u$. It is easy to check $u$ and $F=r^{\frac{(A-n)p+n-A}{2}}|u|^p$ satisfy \eqref{Eq:u_i_r} while $p<p_F$. Thus we get the desired global weak solution.

\part[$p_S<p_t=p$]\label{Pt:2}
In this part we take $k=3$. For $(t,r)\in\Omega$ we find
\begin{Eq*}
\omega_3|u_{j+1}|\leq& \varepsilon C_0(\ln\kl<t-r\kr>)^{-1}\Psi+C_0\|\omega_3 u_j\|_{L_{s,\rho}^\infty(\Lambda)}^p,\\
\omega_3|u_{j+1}-u_j|\leq& C_0\|\omega_3 \kl(u_j-u_{j-1}\kr)\|_{L_{s,\rho}^\infty(\Lambda)}\max_{l\in\{j,j-1\}}\|\omega_3 u_l\|_{L_{s,\rho}^\infty(\Lambda)}^{p-1}.
\end{Eq*}
Noticing $\kl<t-r\kr>\geq 2$, by a similar process as above we get the desired global solution.

\part[$p_S<p<p_t$]\label{Pt:3}
In this part we take $k=2$. For $(t,r)\in\Omega$ we find
\begin{Eq*}
\omega_2|u_{j+1}|\leq& \varepsilon C_0\kl<t-r\kr>^{\frac{(n-1)p-n-A}{2}}\Psi+C_0\|\omega_2 u_j\|_{L_{s,\rho}^\infty(\Lambda)}^p,\\
\omega_2|u_{j+1}-u_j|\leq& C_0\|\omega_2 \kl(u_j-u_{j-1}\kr)\|_{L_{s,\rho}^\infty(\Lambda)}\max_{l\in\{j,j-1\}}\|\omega_2 u_l\|_{L_{s,\rho}^\infty(\Lambda)}^{p-1}.
\end{Eq*}
Here $\frac{(n-1)p-n-A}{2}\leq 0$ since $p<p_t$, by a similar process again we get the desired global solution.

\part[$p<\min(p_d,p_F)$]\label{Pt:4}
In this case, we choose $T_*=T_*(\varepsilon)$ which satisfies $\varepsilon^{p-1}T_*^{\frac{(-n-A+2)p+n+A+2}{2}}=a$ with $a$ to be fixed later. Here we define $\Omega_*:=\Omega\cap \{t:t<T_*\}$ and choose $k=1$. Then for $(t,r)\in\Omega_*$  we find

\begin{Eq*}
\omega_1|u_{j+1}|\leq& \varepsilon C_0\Psi+C_0\kl<t-r\kr>^{\frac{(-n-A+2)p+n+A+2}{2}}\|\omega_1 u_j\|_{L_{s,\rho}^\infty(\Lambda)}^p,\\
\omega_1|u_{j+1}-u_j|\leq& C_0\kl<t-r\kr>^{\frac{(-n-A+2)p+n+A+2}{2}}\|\omega_1 \kl(u_j-u_{j-1}\kr)\|_{L_{s,\rho}^\infty(\Lambda)}\\
&\times\max_{l\in\{j,j-1\}}\|\omega_1 u_l\|_{L_{s,\rho}^\infty(\Lambda)}^{p-1}.
\end{Eq*}
Taking the $L_{t,r}^\infty(\Omega_*)$ norm on both sides, 
when $T_*\geq 3$ we get
\begin{Eq*}
\kl\|\omega_1u_{j+1}\kr\|_{L_{t,r}^\infty(\Omega_*)}\leq& \varepsilon C_1\Psi+C_1T_*^{\frac{(-n-A+2)p+n+A+2}{2}}\|\omega_1 u_j\|_{L_{t,r}^\infty(\Omega_*)}^p,\\
\kl\|\omega_1(u_{j+1}-u_j)\kr\|_{L_{t,r}^\infty(\Omega_*)}\leq& C_1T_*^{\frac{(-n-A+2)p+n+A+2}{2}}\|\omega_1 \kl(u_j-u_{j-1}\kr)\|_{L_{t,r}^\infty(\Omega_*)}\\
&\times\max_{l\in\{j,j-1\}}\|\omega_1 u_l\|_{L_{t,r}^\infty(\Omega_*)}^{p-1},
\end{Eq*}
with some $C_1$ large enough. 
Considering $\varepsilon$ such that $T_*(\varepsilon)\geq 3$ 
and defining $a=(2C_1)^{-p}\Psi^{1-p}$,
we conclude that
\begin{Eq*}
\kl\|\omega_1u_{j}\kr\|_{L_{t,r}^\infty(\Omega_*)}\leq& 2\varepsilon C_1\Psi,\\
\kl\|\omega_1(u_{j+1}-u_j)\kr\|_{L_{t,r}^\infty(\Omega_*)}\leq& \frac{1}{2}\|\omega_1 \kl(u_j-u_{j-1}\kr)\|_{L_{t,r}^\infty(\Omega_*)}
\end{Eq*} 
 for any $j$, which are sufficient to  get the desired solution.

\part[$p=p_F<p_d$]\label{Pt:5}
In this case, 
we choose $T_*$ which satisfies $\varepsilon^{p-1}\ln T_*=a$ 
with $a$ to be fixed later, 
$\Omega_*$ as above and $k=1$.
For $(t,r)\in \Omega_*$ we find
\begin{Eq*}
\omega_1|u_{j+1}|\leq& \varepsilon C_0\Psi+C_0\ln\kl<t-r\kr>\|\omega_1 u_j\|_{L_{s,\rho}^\infty(\Lambda)}^p,\\
\omega_1|u_{j+1}-u_j|\leq& C_0\ln\kl<t-r\kr>\|\omega_1 \kl(u_j-u_{j-1}\kr)\|_{L_{s,\rho}^\infty(\Lambda)}\max_{l\in\{j,j-1\}}\|\omega_1 u_l\|_{L_{s,\rho}^\infty(\Lambda)}^{p-1}.
\end{Eq*}
Similar as above, 
taking $\varepsilon$ such that $T_*(\varepsilon)\geq 3$, 
defining $a:=(2C_1)^{-p}\Psi^{1-p}$ with $C_1$ large enough, 
we get the Cauchy sequence $\{u_j\}$ and the desired solution.

\part[$p=p_d<p_F$]\label{Pt:6}
In this case, 
we choose $T_*$ which satisfies
\begin{Eq}\label{Eq:T*_pd}
\varepsilon^{p-1} T_*^\frac{(-n-A+2)p+n+A+2}{2}\ln T_*=a
\end{Eq}
and $\Omega_*$ same as above. 
Taking $k=1$, 
for $(t,r)\in\Omega_*$  we find
\begin{Eq*}
\omega_1|u_{j+1}|\leq& \varepsilon C_0\Psi+C_0\kl<t-r\kr>^{\frac{(-n-A+2)p+n+A+2}{2}}\ln\kl<t-r\kr>\|\omega_1 u_j\|_{L_{s,\rho}^\infty(\Lambda)}^p,\\
\omega_1|u_{j+1}-u_j|\leq& C_0\kl<t-r\kr>^{\frac{(-n-A+2)p+n+A+2}{2}}\ln\kl<t-r\kr>\|\omega_1 \kl(u_j-u_{j-1}\kr)\|_{L_{s,\rho}^\infty(\Lambda)}\\
&\times\max_{l\in\{j,j-1\}}\|\omega_1 u_l\|_{L_{s,\rho}^\infty(\Lambda)}^{p-1}.
\end{Eq*}
Choosing $\varepsilon$ such that $T_*(\varepsilon)\geq3$, 
$a=(2C_1)^{-p}\Psi^{1-p}$ with $C_1$ large enough, 
we get the Cauchy sequence $\{u_j\}$ and the desired solution.
To finish the proof of \Th{Th:M_1} for this part, 
we introduce the following claim and postpone its proof to the end of this section.
\begin{claim}\label{Cl:T*_pd}
Assume that $T_*$ satisfies $\varepsilon^{p-1} T_*^\frac{-h_F(p)}{2}\ln T_*=a$ with some constant $a$, then there exists two constant $c_1$, $c_2$ such that $c_1\varepsilon^{\frac{2(p-1)}{h_F(p)}}|\ln \varepsilon |^{\frac{2}{h_F(p)}}\leq T_*\leq c_2\varepsilon^{\frac{2(p-1)}{h_F(p)}}|\ln \varepsilon |^{\frac{2}{h_F(p)}}$ for $\varepsilon$ small enough.
\end{claim}

\part[$p=p_d=p_F$]\label{Pt:7}
In this case, 
we choose $T_*$ which satisfies $\varepsilon^{p-1} (\ln T_*)^2=a$ 
and $\Omega_*$ same as above. 
Taking $k=1$, 
for $(t,r)\in\Omega_*$  we find
\begin{Eq*}
\omega_1|u_{j+1}|\leq& \varepsilon C_0\Psi+C_0(\ln \kl<t-r\kr>)^2\|\omega_1 u_j\|_{L_{s,\rho}^\infty(\Lambda)}^p,\\
\omega_1|u_{j+1}-u_j|\leq& C_0(\ln \kl<t-r\kr>)^2\|\omega_1 \kl(u_j-u_{j-1}\kr)\|_{L_{s,\rho}^\infty(\Lambda)}\max_{l\in\{j,j-1\}}\|\omega_1 u_l\|_{L_{s,\rho}^\infty(\Lambda)}^{p-1}.
\end{Eq*}
Choosing $\varepsilon$ such that $T_*(\varepsilon)\geq 3$, 
$a=(2C_1)^{-p}\Psi^{1-p}$ and $C_1$ large enough, 
we get the Cauchy sequence $\{u_j\}$ and finish the proof.

\part[$p_d<p<p_S$]\label{Pt:8}
In this case, 
we choose $T_*$ which satisfies $\varepsilon^{p(p-1)}T_*^{\frac{(1-n)p^2+(n+1)p+2}{2}}=a$ 
with $a$ to be fixed later  and $\Omega_*$ as above. 
Moreover, 
we separate the region $\Omega_*$ to
\begin{Eq*}
\Omega_{*1}:=&\Omega_*\cap\{(t,r):\kl<t-r\kr>\leq (b \varepsilon)^{\frac{2(p-1)}{(n-1)p-n-A}}\},\\
 \Omega_{*2}:=&\Omega_*\cap\{(t,r):\kl<t-r\kr>\geq (b \varepsilon)^{\frac{2(p-1)}{(n-1)p-n-A}}\},
\end{Eq*} 
with $b$ to be fixed later. Firstly we take $k=1$, 
for $(t,r)\in\Omega_{*1}$  we find
\begin{Eq*}
\omega_1|u_{j+1}|\leq& \varepsilon C_0\Psi+C_0\kl<t-r\kr>^{\frac{(1-n)p+n+A}{2}}\|\omega_1 u_j\|_{L_{s,\rho}^\infty(\Lambda)}^p,\\
\omega_1|u_{j+1}-u_j|\leq& C_0\kl<t-r\kr>^{\frac{(1-n)p+n+A}{2}}\|\omega_1 \kl(u_j-u_{j-1}\kr)\|_{L_{s,\rho}^\infty(\Lambda)}\max_{l\in\{j,j-1\}}\|\omega_1 u_l\|_{L_{s,\rho}^\infty(\Lambda)}^{p-1}.
\end{Eq*}
Taking the $L_{t,r}^\infty(\Omega_{*1})$ norm on both sides we get
\begin{Eq*}
\kl\|\omega_1u_{j+1}\kr\|_{L_{t,r}^\infty(\Omega_{*1})}\leq& \varepsilon C_0\Psi+C_0(b\varepsilon)^{1-p}\|\omega_1 u_j\|_{L_{t,r}^\infty(\Omega_{*1})}^p,\\
\kl\|\omega_1(u_{j+1}-u_j)\kr\|_{L_{t,r}^\infty(\Omega_{*1})}\leq& C_0(b\varepsilon)^{1-p}\|\omega_1 \kl(u_j-u_{j-1}\kr)\|_{L_{t,r}^\infty(\Omega_{*1})}\max_{l\in\{j,j-1\}}\|\omega_1 u_l\|_{L_{t,r}^\infty(\Omega_{*1})}^{p-1}.
\end{Eq*}
Choosing $b$ such that $(2C_0)^pb^{1-p}\Psi^{p-1}=1$, we find
\begin{Eq}\label{Eq:u_i_1}
\kl\|\omega_1u_{j}\kr\|_{L_{t,r}^\infty(\Omega_{*1})}\leq& 2\varepsilon C_0\Psi,\\
\kl\|\omega_1(u_{j+1}-u_j)\kr\|_{L_{t,r}^\infty(\Omega_{*1})}\leq& \frac{1}{2}\|\omega_1 \kl(u_j-u_{j-1}\kr)\|_{L_{t,r}^\infty(\Omega_{*1})}
\end{Eq}
holds for any $j$.
This means $\{u_j\}$ is a Cauchy sequence in weighted $L^\infty$ norm on $\Omega_{*1}$ region.
On the other hand, for $(t,r)\in \Omega_{*2}$, we separate $\Lambda$ into $\Lambda_1:=\Lambda\cap \Omega_{*1}$ and $\Lambda_2:=\Lambda\cap \Omega_{*2}$.
Because of the linearity of function, we can separate $u_{j+1}(t,r)$ to two terms which determined by the nonlinear terms supported in $\Lambda_1$ and $\Lambda_2$ respectively.
Taking $k=1$ in the former region and $k=2$ in the latter region, for $(t,r)\in \Omega_{*2}$ we find 
\begin{Eq*}
\omega_2|u_{j+1}|\leq& \varepsilon C_0\kl<t-r\kr>^\frac{(n-1)p-n-A}{2}\Psi+C_0\|\omega_1 u_j\|_{L_{s,\rho}^\infty(\Lambda_1)}^p\\
&+C_0\kl<t-r\kr>^{\frac{(1-n)p^2+(n+1)p+2}{2}}\|\omega_2 u_j\|_{L_{s,\rho}^\infty(\Lambda_2)}^p\\
\omega_2|u_{j+1}-u_j|\leq& C_0\|\omega_1 \kl(u_j-u_{j-1}\kr)\|_{L_{s,\rho}^\infty(\Lambda_1)} \max_{l\in\{j,j-1\}}\|\omega_1 u_l\|_{L_{s,\rho}^\infty(\Lambda_1)}^{p-1},\\
+C_0\kl<t-r\kr>&{}^{\frac{(1-n)p^2+(n+1)p+2}{2}}\|\omega_2 (u_j-u_{j-1})\|_{L_{s,\rho}^\infty(\Lambda_2)}\max_{l\in\{j,j-1\}}\|\omega_2 u_l\|_{L_{s,\rho}^\infty(\Lambda_2)}^{p-1}.
\end{Eq*}
Taking the $L_{t,r}^\infty(\Omega_{*2})$ norm on both sides, 
noticing $\kl<t-r\kr>\geq (b \varepsilon)^{\frac{2(p-1)}{(n-1)p-n-A}}$ for $(t,r)\in \Omega_{*2}$
and $(2C_0)^pb^{1-p}\Psi^{p-1}=1$,
and using the estimate in $\Omega_{*1}$,
when $T_*\geq 3$ we get
\begin{Eq*}
\kl\|\omega_2u_{j+1}\kr\|_{L_{t,r}^\infty(\Omega_{*2})}\leq&C_1(\varepsilon \Psi)^p+C_1T_*^{\frac{(1-n)p^2+(n+1)p+2}{2}}\|\omega_2 u_j\|_{L_{t,r}^\infty(\Omega_{*2})}^p,\\
\kl\|\omega_2(u_{j+1}-u_j)\kr\|_{L_{t,r}^\infty(\Omega_{*2})}\leq& C_1(\varepsilon \Psi)^{p-1}\|\omega_1 \kl(u_j-u_{j-1}\kr)\|_{L_{t,r}^\infty(\Omega_{*1})}\\
+C_1T_*^{\frac{(1-n)p^2+(n+1)p+2}{2}}&\|\omega_2 (u_j-u_{j-1})\|_{L_{t,r}^\infty(\Omega_{*2})}\max_{l\in\{j,j-1\}}\|\omega_2 u_l\|_{L_{t,r}^\infty(\Omega_{*2})}^{p-1},
\end{Eq*}
with some $C_1$ large enough. 
Choosing $\varepsilon$ such that $T_*(\varepsilon)\geq 3$ and $a=(2C_1)^{-p}\Psi^{p(1-p)}$, we find
\begin{Eq*}
\kl\|\omega_2u_{j+1}\kr\|_{L_{t,r}^\infty(\Omega_{*2})}\leq&2C_1(\varepsilon \Psi)^p,\\
\kl\|\omega_2(u_{j+1}-u_j)\kr\|_{L_{t,r}^\infty(\Omega_{*2})}\leq& C_1(\varepsilon \Psi)^{p-1}\|\omega_1 \kl(u_j-u_{j-1}\kr)\|_{L_{t,r}^\infty(\Omega_{*1})}\\
&+\frac{1}{2}\|\omega_2 (u_j-u_{j-1})\|_{L_{t,r}^\infty(\Omega_{*2})},
\end{Eq*}
holds for any $j$. 
Now, since we already know $\{u_j\}$ is a Cauchy sequence on $\Omega_{*1}$ region, 
we can also know $\{u_j\}$ is a Cauchy sequence in weighted $L^\infty$ norm on $\Omega_{*2}$ region.
In summary, we get the desired solution.

\part[$p_d<p=p_S$]\label{Pt:9}
In this part, 
we choose $T_*$ which satisfies $\varepsilon^{p(p-1)}\ln T_*=a$ 
and $\Omega_{*1}$, $\Omega_{*2}$ as above with $a,b$ to be fixed latter. 
Similarly, 
for $(t,r)\in \Omega_{*1}$ we find 
\begin{Eq*}
\omega_1|u_{j+1}|\leq& \varepsilon C_0\Psi+C_0\kl<t-r\kr>^{\frac{(1-n)p+n+A}{2}}\|\omega_1 u_j\|_{L_{s,\rho}^\infty(\Lambda)}^p,\\
\omega_1|u_{j+1}-u_j|\leq& C_0\kl<t-r\kr>^{\frac{(1-n)p+n+A}{2}}\|\omega_1 \kl(u_j-u_{j-1}\kr)\|_{L_{s,\rho}^\infty(\Lambda)}\max_{l\in\{j,j-1\}}\|\omega_1 u_l\|_{L_{s,\rho}^\infty(\Lambda)}^{p-1}.
\end{Eq*}
For $(t,r)\in \Omega_{*2}$ we find 
\begin{Eq*}
\omega_2|u_{j+1}|\leq& \varepsilon C_0\kl<t-r\kr>^\frac{(n-1)p-n-A}{2}\Psi+C_0\|\omega_1 u_j\|_{L_{s,\rho}^\infty(\Lambda_1)}^p\\
&+C_0\ln\kl<t-r\kr>\|\omega_2 u_j\|_{L_{s,\rho}^\infty(\Lambda_2)}^p\\
\omega_2|u_{j+1}-u_j|\leq& C_0\|\omega_1 \kl(u_j-u_{j-1}\kr)\|_{L_{s,\rho}^\infty(\Lambda_1)} \max_{l\in\{j,j-1\}}\|\omega_1 u_l\|_{L_{s,\rho}^\infty(\Lambda_1)}^{p-1}\\
&+C_0\ln\kl<t-r\kr>\|\omega_2 (u_j-u_{j-1})\|_{L_{s,\rho}^\infty(\Lambda_2)}\max_{l\in\{j,j-1\}}\|\omega_2 u_l\|_{L_{s,\rho}^\infty(\Lambda_2)}^{p-1}.
\end{Eq*}
Taking $b$ satisfying $(2C_0)^pb^{1-p}\Psi^{p-1}=1$, 
choosing $\varepsilon$ such that $T_*(\varepsilon)\geq 3$, 
$a=(2C_1)^{-p}\Psi^{p(1-p)}$ and $C_1$ large enough, we get the Cauchy sequence $\{u_j\}$ and the desired solution.

Before the end of this section, we show the proof of \Cl{Cl:T*_pd}. 
Following \eqref{Eq:T*_pd}, we can easily find
\begin{Eq*}
\varepsilon^{\frac{2(p-1)}{h_F(p)}+\delta}\lesssim T_*\lesssim \varepsilon^{\frac{2(p-1)}{h_F(p)}}
\end{Eq*}
for any $0<\delta\ll1$ and $\varepsilon$ small enough. This suggests us to define
\begin{Eq*}
S(\varepsilon):=\varepsilon^{-\frac{2(p-1)}{h_F(p)}}T_*,\qquad \varepsilon^{\delta}\lesssim S\lesssim 1.
\end{Eq*}
Then, \eqref{Eq:T*_pd} goes to
\begin{Eq*}
a=S^\frac{-h_F(p)}{2}\ln\kl(\varepsilon^{\frac{2(p-1)}{h_F(p)}}S\kr)\approx S^\frac{-h_F(p)}{2}|\ln\varepsilon|,
\end{Eq*}
and then
\begin{Eq*}
S\approx |\ln \varepsilon|^{\frac{2}{h_F(p)}},\qquad T_*\approx \varepsilon^{\frac{2(p-1)}{h_F(p)}}|\ln \varepsilon|^{\frac{2}{h_F(p)}},
\end{Eq*}
which finishes the proof.

\section{Long-time existence for $A\in[3,\infty)$}\label{Se:4}
In this section, we will consider the case $A\in[3,\infty)$, and show the proof of \Th{Th:M_2}. Again, we only need to consider the equation \eqref{Eq:u_o}.

\subsection{Estimate for linear solution}
In this subsection, we will construct some prior estimates of the solution to the linear equation \eqref{Eq:u_l}. Firstly we give the following argument.
\begin{lemma}\label{Le:ul_e1}
Let $u$ be the solution of \eqref{Eq:u_l}. We have
\begin{Eq}\label{Eq:ul_e11}
\kl\|r^{\frac{A-1}{2}}u\kr\|_{L_t^\infty L_r^q}\lesssim& \|r^{\frac{A+1}{2}}g\|_{L_r^q}+ \|r^{\frac{A-1}{2}}f\|_{L_r^q}+\|r^{\frac{A+1}{2}}F\|_{L_t^qL_r^1},
\end{Eq}
for any $1<q<\infty$, and 
\begin{Eq}\label{Eq:ul_e12}
\kl\|r^{\frac{A-1}{2}-\alpha}u\kr\|_{L_t^\sigma L_r^p}\lesssim& \|r^{\frac{A+1}{2}}g\|_{L_r^q}+ \|r^{\frac{A-1}{2}}f\|_{L_r^q}+\|r^{\frac{A+1}{2}}F\|_{L_t^qL_r^1},
\end{Eq}
provided that
\begin{Eq}\label{Eq:sigpq_r}
1<\frac{q}{p}<\frac{\sigma}{p}<\infty,\qquad~\alpha p=1-\frac{p}{q}+\frac{p}{\sigma}.
\end{Eq}
\end{lemma}
We also have a modification result.
\begin{lemma}\label{Le:ul_e2}
Let $u$ be the solution of \eqref{Eq:u_l}. If $1<p\leq (n+1)/(n-1)$, then
\begin{Eq}\label{Eq:ul_e21}
&(T+1)^{\frac{(n-1)p-n-1}{2p}}\kl\|r^{\frac{(A-n)p+n+1}{2p}}u(T,r)\kr\|_{L_r^p(r<T+1)}\\
\lesssim& \|r^{\frac{A+1}{2}}g\|_{L_r^p}+ \|r^{\frac{A-1}{2}}f\|_{L_r^p}+\|r^{\frac{A+1}{2}}F\|_{L_t^pL_r^1(t<T)}.
\end{Eq}
Also, 
if $(n+1)/(n-1)\leq p\leq p_S$, 
then
\begin{Eq}\label{Eq:ul_e22}
&T^{\frac{(n-1)p-n-1}{2p}}\kl\|r^{\frac{(A-n)p+n+1}{2p}}u(T,r)\kr\|_{L_r^p}\\
\lesssim& \|r^{\frac{A+1}{2}+\frac{1}{p}}g\|_{L_r^\infty}+ \|r^{\frac{A-1}{2}+\frac{1}{p}}f\|_{L_r^\infty}+T^{\frac{1}{p}}\|r^{\frac{A+1}{2}}F\|_{L_t^\infty L_r^1(T/4<t<T)}\\
&+\|r^{\frac{A+1}{2}}g\|_{L_r^p}+ \|r^{\frac{A-1}{2}}f\|_{L_r^p}+\|r^{\frac{A+1}{2}}F\|_{L_t^pL_r^1(t<T/4)}.
\end{Eq}
\end{lemma}

\begin{proof}[Proof of \Le{Le:ul_e1}]
When $A\in \SZ_+$, 
this result is exactly the same with the Theorem 4.7 of \cite{MR1408499} 
since we can take $\kappa$ in Theorem 4.7 arbitrary close to $2$. 
So we only deal with the non-integer case.

Here we introduce $\delta< (A-3)/2$ which will be fixed later. 
Firstly we consider $u=u_g$.
Using \Le{Le:u_e} and \Le{Le:I_p}  with $\mu=\frac{r^2+\rho^2-t^2}{2r\rho}$,
we have
\begin{Eq*}
r^{\frac{A-1}{2}}|u_g|\lesssim&
\int_{|t-r|}^{t+r}(1+\mu)^{-\delta}\rho^{\frac{A-1}{2}}|g(\rho)|\d\rho+\int\limits_{0<\rho<t-r\atop -2<\mu<-1}|1+\mu|^{-\delta}\rho^{\frac{A-1}{2}}|g(\rho)|\d\rho\\
&+\int\limits_{0<\rho<t-r\atop \mu<-2}(1-\mu)^{\frac{1-A}{2}}\rho^{\frac{A-1}{2}}|g(\rho)|\d\rho\\
\lesssim&\int_{|t-r|}^{t+r}(1+\mu)^{-\delta}\rho^{-1}|\rho^{\frac{A+1}{2}}g(\rho)|\d\rho\\
&+\int\limits_{0<\rho<t-r}(1-\mu)^{-1}|1+\mu|^{-\delta}\rho^{-1}|\rho^{\frac{A+1}{2}}g(\rho)|\d\rho,
\end{Eq*}
where the second integral does not appear for $t<r$.
Using Proposition 2.7 and Proposition 4.4 in \cite{MR1408499}, we obtain the estimate as we desired.

Next we consider $u=u_f$. Here for simplicity we denote $h(\rho):=\rho^{\frac{A-1}{2}}f(\rho)$.
Using \Le{Le:u_e} and \Le{Le:I_p} again with $\mu=\frac{r^2+\rho^2-t^2}{2r\rho}$,
for $r<t$ we find
\begin{Eq*}
r^{\frac{A-1}{2}}u_f=&\frac{1}{2}h(t+r)-P.V.\int_0^{t+r}\frac{t}{r\rho}I_A'(\mu)h(\rho)\d\rho,\\
r^{\frac{A-1}{2}}|u_f|\lesssim& \frac{1}{2}|h(t+r)|+\kl|P.V.\int\limits_{0\leq\rho\leq t+r\atop -2<\mu<1}\frac{t}{r\rho}(1+\mu)^{-1}h(\rho)\d\rho\kr|\\
&+\int\limits_{0\leq\rho\leq t+r\atop -2<\mu<1}\frac{t}{r\rho}|1+\mu|^{-\delta}|h(\rho)|\d\rho+\int\limits_{0\leq\rho\leq t-r\atop \mu<-2}\frac{t}{r\rho}|1-\mu|^{\frac{-1-A}{2}}|h(\rho)|\d\rho\\
\lesssim&\frac{1}{2}|h(t+r)|+\kl|P.V.\int_{0}^{t+r}\frac{t}{r\rho}(1+\mu)^{-1}h(\rho)\d\rho\kr|\\
&+\int_{0}^{t-r}\frac{t}{r\rho}|1-\mu|^{-1}|h(\rho)|\d\rho+\int_{t-r}^{t+r}\frac{t}{r\rho}|1+\mu|^{-\delta}|h(\rho)|\d\rho\\
&+\int_0^{t-r}\frac{t}{r\rho}(1-\mu)^{-1}|1+\mu|^{-\delta}|h(\rho)|\d\rho\\
\equiv& K_1+K_2+K_3+K_4+K_5.
\end{Eq*}
It's easy to find that
\begin{Eq*}
\|K_1\|_{L_t^\infty L_r^q}\lesssim \|h\|_{L_r^q}.
\end{Eq*}
On the other, we introduce the well known \emph{Hardy-Littlewood} inequality
\begin{Eq*}
\kl\||y|^{-\alpha}f(x\pm y)\kr\|_{L_x^\sigma L_y^p(\SR^2)} \lesssim \|f\|_{L^q(\SR)},
\end{Eq*}
with \eqref{Eq:sigpq_r}. Now, taking $f(x)=|h(x)|\chi_{[0,\infty)}(x)$, we also find
\begin{Eq*}
\|r^{-\alpha}K_1\|_{L_t^\sigma L_r^p}\lesssim \|h\|_{L_r^q},
\end{Eq*}
provided that \eqref{Eq:sigpq_r}.
Meanwhile, adopting Proposition 2.5 in \cite{MR1408499} to deal with $K_4$, and adopting Proposition 4.4 in \cite{MR1408499} for $K_3$ and $K_5$, we find both of them have the same control as $K_1$.  
As for $K_2$, noticing
\begin{Eq*}
\frac{t}{r\rho}(1+\mu)^{-1}=\frac{1}{r+\rho-t}-\frac{1}{t+r+\rho},
\end{Eq*}
we can control $K_2$ by
\begin{Eq*}
K_2\lesssim &\kl|P.V.\int_0^\infty \frac{1}{\rho+r-t}h(\rho)\d\rho\kr|+\int_{t+r}^\infty \frac{1}{\rho+r-t}|h(\rho)|\d\rho\\
&+\frac{1}{t+r}\int_0^{t+r}|h(\rho)|\d\rho\\
\equiv&K_{2,1}(t-r)+K_{2,2}(t,r)+K_{2,3}(t+r).
\end{Eq*}
For $K_{2,1}$, we introduce the estimate of \emph{Hilbert}-transform
\begin{Eq*}
\kl\|P.V.\int \frac{1}{x-y}f(y)\d y\kr\|_{L_x^q}\lesssim \|f\|_{L_x^q}
\end{Eq*}
with $1<q<\infty$. Taking $f(x)=h(x)\chi_{[0,\infty)}(x)$, we find
\begin{Eq*}
\kl\|K_{2,1}(t-r)\kr\|_{L_t^\infty L_r^q}=\kl\|K_{2,1}(r)\kr\|_{L_r^q}\lesssim \|h\|_{L^q}.
\end{Eq*}
Also, using \emph{Hardy-Littlewood} inequality again, we get the dominate of $K_{2,1}(t-r)$ same as that of $K_1$ provided that \eqref{Eq:sigpq_r}. As for $K_{2,3}$, we introduce the \emph{Hardy-Littlewood} maximal inequality
\begin{Eq*}
\kl\|\sup_{y>0}\frac{1}{2y}\int_{x-y}^{x+y} f(z)\d z\kr\|_{L_x^q}\lesssim \|f\|_{L^q}
\end{Eq*}
with $1<q<\infty$. Taking $f(x)=h(x)\chi_{[0,\infty)}(x)$, we find
\begin{Eq*}
\kl\|K_{2,3}(t+r)\kr\|_{L_t^\infty L_r^q}=\kl\|K_{2,3}(r)\kr\|_{L_r^q}\lesssim \|h\|_{L^q}.
\end{Eq*}
Using \emph{Hardy-Littlewood} inequality again, we get the dominate of $K_{2,3}(t+r)$ same as that of $K_1$ provided that \eqref{Eq:sigpq_r}.
Finally, for $K_{2,2}(t,r)$, we have
\begin{Eq*}
K_{2,2}(t,r)=\int_r^\infty \frac{1}{\rho+r}|h(\rho+t)|\d\rho\leq \int_r^\infty \rho^{-1}|h(\rho+t)|\d\rho.
\end{Eq*}
Then, using \emph{Hardy's} inequality we find
\begin{Eq*}
\|K_{2,2}\|_{L_t^\infty L_r^q}\lesssim \|h\|_{L_r^q}.
\end{Eq*}
On the other hand, for any $G(t,r)$ with $\|G\|_{L_t^{\sigma'}L_r^{p'}}\leq 1$ we see
\begin{Eq*}
&\int_0^\infty\int_0^\infty r^{-\alpha}K_{2,2}(t,r)G(t,r)\d r\d t\\
=&\int_0^\infty\int_t^\infty \int_{0}^{\rho-t} \frac{r^{-\alpha}}{\rho+r-t}|h(\rho)|G(t,r)\d r\d\rho\d t\\
\lesssim &\int_0^\infty\int_t^\infty |h(\rho)|\kl\|\frac{r^{-\alpha}}{\rho+r-t} \kr\|_{L_r^p(0,\rho-t)}\|G\|_{L_r^{p'}}\d\rho\d t\\
\approx &\int_0^\infty\int_t^\infty |h(\rho)||\rho-t|^{-\alpha-1+\frac{1}{p}}\|G\|_{L_r^{p'}}\d\rho\d t\\
\lesssim &\kl\|\int_t^\infty |h(\rho)||\rho-t|^{-\alpha-1+\frac{1}{p}}\d\rho\kr\|_{L_t^{\sigma}}\\
\lesssim &\|h\|_{L_r^q},
\end{Eq*}
where in the last step we use the \emph{Hardy-Littlewood} inequality. Now, we find
\begin{Eq*}
\|r^{-\alpha}K_{2,2}\|_{L_t^\sigma L_r^p}\leq\sup\limits_{\|G\|_{L_t^{\sigma'}L_r^{p'}}\leq 1}\kl<r^{-\alpha}K_{2,2},G\kr>\lesssim \|h\|_{L_r^q}.
\end{Eq*}
Mixing these results, we obtain the estimate for $r<t$ part. For $r>t$, we have 
\begin{Eq*}
r^{\frac{A-1}{2}}u_f=&\frac{1}{2}h(t+r)+\frac{1}{2}h(r-t)-\int_{r-t}^{t+r}\frac{t}{r\rho}I_A'(\mu)h(\rho)\d\rho\\
\equiv &K_1'+K_2'+K_3'.
\end{Eq*}
The estimate of $K_1'$ and $K_2'$ is the same as that of $K_1$ in $r<t$ part, and the estimate of $K_3'$ is the same as that of $K_4$. Adding all together, we finish the proof of $u_f$ part.

Finally, we consider $u=u_F$. 
Using \Le{Le:u_e} and \Le{Le:I_p} again with $\mu=\frac{r^2+\rho^2-(t-s)^2}{2r\rho}$,
similarly we have
\begin{Eq*}
r^{\frac{A-1}{2}}u_F=&\int_0^t\int_{0}^{r+t-s}I_A(\mu)\rho^{\frac{A-1}{2}}F(s,\rho)\d\rho\d s\\
r^{\frac{A-1}{2}}|u_F|\lesssim&\int_0^t\int_{|r-t+s|}^{r+t-s}(1+\mu)^{-\delta}\rho^{-1}|G(s,\rho)|\d\rho\d s
\\
&+\iint\limits_{0\leq s\leq t\atop 0\leq \rho\leq t-s-r}(1-\mu)^{-1}|1+\mu|^{-\delta}\rho^{-1}|G(s,\rho)|\d\rho\d s
\end{Eq*}
where $G(s,\rho):=\rho^{\frac{A+1}{2}}F(s,\rho)$ and the second integral does not appear for $t-s<r$. Here we choose $\delta$ small enough. Using Proposition 4.5 in \cite{MR1408499}, we obtain the desired estimate for $u_F$. 
Now, we finish the proof of \Le{Le:ul_e1}.
\end{proof}

\begin{proof}[Proof of \Le{Le:ul_e2}]
The proof of  \Le{Le:ul_e2} is almost the same with that of Theorem 6.4 in \cite{MR1408499}. Thus, we only give a sketch of the proof.  
The estimate \eqref{Eq:ul_e21} and part of estimate \eqref{Eq:ul_e22} are direct consequence of \eqref{Eq:ul_e11} with $q=p$.
We only need to show the estimate of 
$T^{\frac{(n-1)p-n-1}{2p}}\kl\|r^{\frac{(A-n)p+n+1}{2p}}u(T,r)\kr\|_{L_r^p(r<T/4)}$. 
To dominate $u_f$, we separate $f=f_0+f_1$ with $f_0=\chi_{[0,T/4]}f$. 
Then, $u_{f=f_1}$ depends on $f_1(\rho)$ with $T/4<\rho<5T/4$. 
So, we obtain the weight of $T$ by extracting the weight of $\rho$. 
For $u_{f=f_0}$, we find that $\rho,r<T/4$ in the expression of $u_{f=f_0}$ 
where $|1\pm\mu|^{-1}\approx r\rho/T^2$.
Then we get the desired estimate by a direct calculation.

The estimate of $u_g$ is similar to that of $u_f$. Finally for $u_F$, we separate the integral in $u_F$ into three parts: $\{r\geq(T-s)/4\}$, $\{r,\rho\leq(T-s)/4\}$ and $\{r\leq (T-s)/4\leq \rho\}$. Then we get the estimate by a similar discussion.
\end{proof}

\subsection{Long-time existence for $1<p<p_{conf}$}
In this subsection, we will give the proof of \Th{Th:M_2}.
The main process of proof is almost the same as that of \cite[Theorem 5.1, Theorem 6.1 and Theorem 6.3]{MR1408499}.
So we only prove global existence in $p_S<p<p_{conf}$ and long-time existence in $p_m\leq p<p_S$ to show such processes fit our frame.

\setcounter{part0}{0}
\part[Proof of $p_S< p<p_{conf}$]
Similar to the last section, 
we will construct a Cauchy sequence to approach the weak solution. 
We set $u_{-1}=0$ and let $u_{j+1}$ be the solution of the equation \eqref{Eq:uj_o}.

We are going to use the estimate \eqref{Eq:ul_e12}, 
where we set
\begin{Eq*}
q=\frac{2(p-1)}{(n+3)-(n-1)p},\qquad \sigma=pq,\qquad \alpha=\frac{(n-1)p-n-1}{2p}.
\end{Eq*}
Here $p<p_{conf}$ so $0<q<\infty$ and $p>p_S$ so $q>p$. Then, we conclude
\begin{Eq*}
\kl\|r^{\frac{(A-n)p+n+1}{2p}}u_{j+1}\kr\|_{L_t^{pq} L_r^p}\leq& \varepsilon C_0\Psi+C_0\|r^{\frac{(A-n)p+n+1}{2}}|u_{j}|^p\|_{L_t^qL_r^1}\\
\leq& \varepsilon C_0\Psi+C_0\|r^{\frac{(A-n)p+n+1}{2p}}u_{j}\|_{L_t^{pq}L_r^p}^p\\
\Psi:=&\kl\|r^\frac{n+1}{2}U_1\kr\|_{L_r^q}+\|r^{\frac{n-1}{2}}U_0\|_{L_r^q}\\
\end{Eq*}
for some $C_0$ large enough. 
Then, for any $\varepsilon$ satisfies $(2\varepsilon C_0\Psi)^p<\varepsilon \Psi$ we find 
\begin{Eq*}
\kl\|r^{\frac{(A-n)p+n+1}{2p}}u_{j}\kr\|_{L_t^{pq} L_r^p}\leq 2\varepsilon C_0\Psi
\end{Eq*}
holds for any $j\geq 0$. By this result and \eqref{Eq:ul_e12} we also find
\begin{Eq*}
&\kl\|r^{\frac{(A-n)p+n+1}{2p}}\kl(u_{j+1}-u_j\kr)\kr\|_{L_t^{pq} L_r^p}\\
\leq& C_0\|r^{\frac{(A-n)p+n+1}{2}}\kl(|u_{j}|^p-|u_{j-1}|^p\kr)\|_{L_t^qL_r^1}\\
\leq& C_1\|r^{\frac{(A-n)p+n+1}{2p}}(u_{j}-u_{j-1})\|_{L_t^{pq}L_r^p}\max_{k\in\{j,j-1\}}\|r^{\frac{(A-n)p+n+1}{2p}}u_k\|_{L_t^{pq}L_r^p}^{p-1}\\
\leq& C_1(2\varepsilon C_0\Psi)^{p-1}\|r^{\frac{(A-n)p+n+1}{2p}}(u_{j}-u_{j-1})\|_{L_t^{pq}L_r^p}
\end{Eq*}
with some $C_1$ large enough. 
Now, for any $0<\varepsilon$ with $C_1(2\varepsilon C_0\Psi)^{p-1}<1/2$, 
we find $\{u_j\}$ is a Cauchy sequence in its space. 
Set the limit as $u$.
Using \eqref{Eq:ul_e11} we can also find 
$\|r^{\frac{A-1}{2}}u\|_{L_t^\infty L_r^q}\leq 2\varepsilon C_2\Psi$ with some $C_2$. 
To check it is the weak solution of \eqref{Eq:u_o} indeed, 
we need to show \eqref{Eq:u_i_r}.
For any compact set $K_t\times K_r\subset \SR_+^2$, we find
\begin{Eq*}
\kl\|r^{\frac{(A-n)p+n+A}{2}-1}|u|^p\kr\|_{L_{t,r}^1(K_t\times K_r)}\leq& C(K_t)\kl\|r^{\frac{(A-n)p+n+1}{2p}}u\kr\|_{L_t^{pq} L_r^p}^p\kl\|r^{\frac{A-3}{2}}\kr\|_{L_r^{\infty}(K_r)}\\
\leq& (2\varepsilon C_0\Psi)^pC(K_t,K_r)<\infty\\
\|r^{A-1}u\|_{L_{t,r}^1(K_t\times K_r)}\leq &C(K_t)\kl\|r^{\frac{A-1}{2}}u\kr\|_{L_t^\infty L_r^q}\kl\|r^{\frac{A-1}{2}}\kr\|_{L_r^{q'}(K_r)}\\
\leq& (2\varepsilon C_1\Psi)C(K_t,K_r)<\infty.
\end{Eq*}
This finishes the proof.

\part[Proof of $p_m\leq p<p_S$]
Next we consider $(n+1)/(n-1)\leq p< p_S$. Set
\begin{Eq*}
A_j(T):=T^{\frac{(n-1)p-n-1}{2p}}\kl\|r^{\frac{(A-n)p+n+1}{2p}}u_{j}(T,r)\kr\|_{L_r^p},
\end{Eq*}
using \eqref{Eq:ul_e22} we see
\begin{Eq*}
A_{j+1}(T)\leq& \varepsilon C_0\Psi+ C_0T^{\frac{1}{p}}\|r^{\frac{(A-n)p+n+1}{2}}|u_j|^p\|_{L_t^\infty L_r^1(T/4<t<T)}
\\
&+C_0\|r^{\frac{(A-n)p+n+1}{2}}|u_j|^p\|_{L_t^pL_r^1(t<T/4)}\\
\leq& \varepsilon C_0\Psi+ C_0T^{\frac{1}{p}}\|t^{\frac{(1-n)p+n+1}{2p}}A_j(t)\|_{L_t^\infty(T/4<t<T)}^p
\\
&+C_0\|t^{\frac{(1-n)p+n+1}{2p}}A_j(t)\|_{L_t^{p^2}(t<T/4)}^p
\end{Eq*}
for some $C_0$ large enough and 
\begin{Eq*}
\Psi:=\|r^{\frac{n+1}{2}+\frac{1}{p}}U_1\|_{L_r^\infty}+ \|r^{\frac{n-1}{2}+\frac{1}{p}}U_0\|_{L_r^\infty}
&+\|r^{\frac{u+1}{2}}U_1\|_{L_r^p}+ \|r^{\frac{n-1}{2}}U_0\|_{L_r^p}.
\end{Eq*}
When $\sup_{0\leq T\leq T_*}A_j(T)\leq 2\varepsilon C_0\Psi$ holds for some $j$ and $T_*$ defined in \eqref{Eq:Main_6} with $c$ to be fixed later, we find
\begin{Eq*}
\sup_{0\leq T\leq T_*} A_{j+1}(T)\leq&\varepsilon C_0\Psi+ C_1(2\varepsilon C_0 \Psi)^pT_*^{\frac{(1-n)p^2+(n+1)p+2}{2p}}\\
\leq&\varepsilon C_0\Psi+ \varepsilon C_1(2 C_0 \Psi)^p c^{\frac{-h_S}{2p}}
\end{Eq*}
with some $C_1$. Choosing $c$ small enough such that $C_1(2 C_0 \Psi)^p c^{\frac{-h_S}{2p}}\leq C_0\Psi$ and noticing $A_{-1}(T)\equiv 0$, we find such estimate holds for any $j$. A similar manner also shows 
\begin{Eq*}
\sup_{0\leq T\leq T_*} B_{j+1}(T)\leq& \frac{1}{2}\sup_{0\leq T\leq T_*} B_{j}(T),\\
B_{j}(T):=&T^{\frac{(n-1)p-n-1}{2p}}\kl\|r^{\frac{(A-n)p+n+1}{2p}}\kl(u_{j}-u_{j-1}\kr)(T,r)\kr\|_{L_r^p}.
\end{Eq*}
Then, we get the desired solution $u$ as the limit of $\{u_j\}$.
Now we check \eqref{Eq:u_i_r}. For any compact set $K_t\times K_r\subset \SR_+^2$, we find
\begin{Eq*}
\kl\|r^{\frac{(A-n)p+n+A}{2}-1}|u|^p\kr\|_{L_{t,r}^1(K_t\times K_r)}\leq& \kl\|t^{\frac{(n-1)p-n-1}{2p}}r^{\frac{(A-n)p+n+1}{2p}}u\kr\|_{L_t^\infty L_r^p}^p\\
&\times \kl\|t^{\frac{(1-n)p+n+1}{2}}r^{\frac{A-3}{2}}\kr\|_{L_t^1 L_r^\infty(K_t\times K_r)},\\
\|r^{A-1}u\|_{L_{t,r}^1(K_t\times K_r)}\leq &\kl\|t^{\frac{(n-1)p-n-1}{2p}}r^{\frac{(A-n)p+n+1}{2p}}u\kr\|_{L_t^\infty L_r^p}\\
&\times\kl\|t^{\frac{(1-n)p+n+1}{2p}}r^{\frac{(n+A-2)p-n-1}{2p}}\kr\|_{L_t^1 L_r^{p'}(K_t\times K_r)}.
\end{Eq*}
Noticing $\frac{(n+A-2)p-n-1}{2p}>\frac{A-3}{2p}\geq 0$ and $\frac{(1-n)p+n+1}{2}>-1$ due to $p<p_S<p_{conf}$, we find both of the above two terms are finite.
This finishes the proof.

\subsection*{Acknowledgments}
The authors would like to thank the anonymous referee for the careful reading and valuable comments.
The authors were supported by NSFC 11671353 and NSFC 11971428.

\end{document}